\documentclass[11pt,reqno]{article}
\usepackage[margin=1in]{geometry}
\usepackage{dsfont, amssymb,amsmath,amscd,latexsym, amsthm, amsxtra,amsfonts}
\usepackage{lineno}
\usepackage[all]{xy}
\usepackage[active]{srcltx}

\usepackage{tikz}
\usepackage[round]{natbib}
\usepackage{bbm}
\usepackage{enumerate}
\usepackage{mathrsfs}
\usepackage{graphicx}
\usepackage{comment}
\usepackage{mathtools}
\usepackage{cases}
 \usepackage{tikz}
\usetikzlibrary{calc,arrows}
\usepackage{verbatim}
\usepackage{graphicx}
\usepackage{subfigure}
\usepackage{color}

\usepackage{tikz}
\usetikzlibrary{calc,arrows}
\usepackage{verbatim}
\usepackage{graphicx}
\usepackage{subfigure}
\usepackage{epstopdf}
\usepackage{verbatim}
\usepackage{graphicx}
\usepackage{subfigure}
\usepackage{color}
\usepackage[affil-it]{authblk}

\newtheorem{theorem}{Theorem}[section]
\newtheorem{corollary}[theorem]{Corollary}
\newtheorem{proposition}[theorem]{Proposition}

\theoremstyle{definition}
\newtheorem{definition}[theorem]{Definition}

\theoremstyle{definition}

\theoremstyle{definition}
\newtheorem{remark}{Remark}
\theoremstyle{definition}
\newtheorem{assumption}{Assumption}

\renewcommand{\epsilon}{\varepsilon}

\usepackage[pdfstartview=FitH, bookmarksnumbered=true,bookmarksopen=true, colorlinks=true, pdfborder=001, citecolor=blue, linkcolor=blue,urlcolor=blue]{hyperref}
\usepackage{graphics}
\graphicspath{{figures/}}
\title{Equilibria for Time-inconsistent Singular Control Problems}

\author{Zongxia Liang$^a$\thanks{Email: \texttt{liangzongxia@mail.tsinghua.edu.cn}}\ \ \ \ \  Xiaodong Luo$^a$\thanks{Email: \texttt{luoxd21@mails.tsinghua.edu.cn}}\ \ \ \ \  Fengyi Yuan$^a$\thanks{Email: \texttt{yfy19@mails.tsinghua.edu.cn}}
 }	
\affil{$^a$Department of Mathematical Sciences, Tsinghua University, China}
\numberwithin{equation}{section}

\usepackage{amsopn}

\begin{document}

\maketitle

\begin{abstract}
We study a time-inconsistent singular control problem originating from irreversible reinsurance decisions with non-exponential discount. A novel definition of equilibrium for time-inconsistent singular control problems is introduced. For the problem with non-exponential discount, both sufficient and necessary conditions are derived, providing a thorough mathematical characterization of the equilibrium. Specifically, the equilibrium can be characterized by an extended HJB system, which is a coupled system of non-local parabolic equations with free boundaries. Finally, by showing the existence of classical solutions to the extended HJB system, the existence of equilibria is established under some additional assumptions.
\vskip 5pt\noindent
\end{abstract}

\section{Introduction}
Ever since the research \cite{benevs1980some} on fuel follower problems, the singular control theory has developed rapidly and become a heated topic in stochastic control. In contrast with regular stochastic control, the singular control is a non-decreasing c$\acute{a}$dl$\acute{a}$g process. The monotonicity of singular control makes it suitable for the modeling of many financial and economic problems such as optimal dividend (\cite{chen2021free}, \cite{qiu2023optimal}), irreversible investment (\cite{kogan2001equilibrium}), optimal extraction (\cite{ferrari2021optimal}) and interbank lending (\cite{cont2021interbank}). Recently, the seminal work by \cite{yan2022irreversible} applies the singular control theory to an irreversible reinsurance problem. 
In this formulation, there is an insurance company deciding on a reinsurance strategy. The insurance company's risk exposure is the state variable representing the insured loss that the company is faced with. The risk exposure is stochastic according to the Cram\'{e}r–Lundberg risk model (\cite{cramer1930mathematical}, \cite{lundberg1926forsakringsteknisk}) of the insured loss, which has a diffusion approximation (see \cite{luo2008reinsurance} and references therein). Reinsurance purchases can lower the company's risk exposure at a cost and are irreversible. The accumulated lowered risk exposure exhibits a monotone feature and can be presented as a singular control. The company's objective is to minimize the combination of the running risk loss, the terminal risk loss and the reinsurance cost over a finite-time horizon.

To the best of our knowledge, despite extensive literature on singular control problems, there is no general framework to accommodate time-inconsistency, which has drawn more and more attention in recent years. To tackle this, we study a general singular control problem with non-exponential discount, which includes aforementioned irreversible reinsurance problems as special cases.
Led by \cite{strotz1973myopia}, there is fruitful research on time-inconsistent problems including time-inconsistent regular control and time-inconsistent stopping (\cite{hu2012time}, \cite{huang2021strong}, \cite{huang2020optimal}). In this paper, we focus on singular control problems and only consider the time-inconsistency brought by non-exponential discount, which is one of the major causes of time-inconsistency in existing literature (\cite{ekeland2008investment}, \cite{harris2001dynamic}). However, our definition of equilibrium can readily be applied to other types of time-inconsistency in singular control problems. 

As in some existing time-inconsistent control literature, we aim to find time-consistent equilibrium solutions. However, bringing together the singular control and time-inconsistency, even the definition of equilibrium is not straightforward. The classical approach to defining equilibrium in time-inconsistent regular control problems is to compare the candidate control and the perturbed control (see \cite{bjork2021time}). For singular control problems, it is difficult to define the perturbed control with the needed properties. In the existing literature, there are two major approaches. The first one is to constrain the dividend process as integral of the dividend rate process and apply the standard approach for regular control problems (\cite{li2016equilibrium}, \cite{zhao2014dividend}). Essentially, this reduces to a time-inconsistent regular control problem. The second approach is to seek an intra-personal sub-game perfect equilibrium (\cite{chen2014optimal}, \cite{zhu2020singular}). Nevertheless, this approach requires discretization of future selves and an additional parameter for the exponential distribution of future selves' duration. As a result, the formulation is completely different from that of time-inconsistent regular control problems. Thus, there is limited literature studying time-inconsistent singular control problems in a way parallel to regular control problems. As an exception, \cite{christensen2022moment} defines an equilibrium by the classical first-order condition under perturbation. However, their results crucially rely on the structure of dividend problems, and due to some dividend constraints, the strategies are limited to impulse control type. We make contributions in this direction by proposing a novel definition of equilibrium. In our formulation, \textit{singular controls} are generated from \textit{singular control laws} by solving the Skorohold reflection problems, and the singular control laws are determined by waiting and purchasing regions; see Definitions \ref{ads} and \ref{adml}. The equilibria are then defined to be those singular control laws that are locally optimal under certain perturbations. Notably, the perturbations are considered on the generated singular controls, instead of the laws themselves; see Definitions \ref{equil}. 

Mathematical characterizations of the equilibrium possess features from both key aspects of our model: the singular control formulation and the time-inconsistency. {\it First}, the jump of singular control requires special treatment on terminal conditions, which corresponds to the terminal optimality of equilibrium.  {\it Second}, the resulting characterizing system consists of two coupled non-local parabolic PDEs with free boundaries. Importantly, global regularity is not required in our verification theorem, which paves the way for our existence result.

Another major contribution of this paper is the existence of equilibria in a fairly general setting. The proof is based on the existence of solutions to the system characterizing the equilibrium. To the best of our knowledge, this is the first theoretical existence result of equilibrium solutions to a singular control problem with time-inconsistency. Usually, time-inconsistency requires a so-called auxiliary value function to be solved simultaneously with the value function, making the extended HJB equations a coupled system. However, with singular control problems, not only the functions but also the regions are coupled, and these regions are endogenously determined by the solution; see $W$ and $P$ in (\ref{HJB}). All these lead to a complicated coupled system with non-linearity, non-locality and free boundaries. To address these challenges, we decouple the system into two sub-problems and apply the fixed point theorem to obtain a solution to the original problem. Two sub-problems are studied using dedicated bounded region approximation, penalty method and Sobolev estimations; see Subsections \ref{subproblem1} and \ref{subproblem2}. A key point of our proof is a tailor-made function space, so that the existence and certain uniqueness of sub-problems are guaranteed; see $\mathcal{R}_0$ in Subsection \ref{subproblem1}. {\it A priori} estimations are obtained by-products, which in turn validate the use of Schauder's fixed point theorem; see Subsection \ref{fixpoint}.   

We would like to address several key differences between the formulation of this paper and that of \cite{yan2022irreversible}, which is the most related literature. First, the problem in \cite{yan2022irreversible} is time-consistent while the current work is time-inconsistent due to the non-exponential discount function. This brings about substantial technical challenges in our paper, especially when proving the existence of equilibria; see the constraint $\mathcal{R}_0$ on $f$ in Subsection \ref{subproblem1}, as well as the a priori estimate of $f$ in Theorem \ref{exist-f}. Second, \cite{yan2022irreversible} studies an infinite-time problem while the time horizon of this work is finite. This feature leads to a system of PDEs instead of ODEs, and the latter admits closed-form solutions in \cite{yan2022irreversible}. Another difference is that the singular control in \cite{yan2022irreversible} has an upper bound $\bar{y}$ representing the risk retention level, which makes it a finite-fuel problem. The current work does not consider this constraint.

The rest of the paper is organized as follows. In Section \ref{modfor}, we state the setting of our time-inconsistent singular control problem and introduce the definition of equilibrium. Sufficient conditions are established in Section \ref{suf} while necessary conditions are in Section \ref{nec}. Finally, in Section \ref{exi}, with some additional assumptions, we prove the existence of equilibria. Section \ref{concl} concludes the paper.

\section{Model Formulation}
\label{modfor}
\subsection*{Notation}
$C^{i}(U)$ is the space of functions on $U$ that are $i$th order continuously differentiable. $C^{i,j}(U), C^{i,j,k}(U), C^{i,j,k,l}(U)$ can be defined similarly if the functions have more than one variable. $\overline{U}$ is the closure of $U$. $C^{k+\alpha,\frac{k+\alpha}{2}}(\overline{U})$ denotes the standard H{\"o}lder space and $C^{k+\alpha,\frac{k+\alpha}{2}}(U)$ denotes the space of all functions $u\in C^{k+\alpha,\frac{k+\alpha}{2}}(\overline{W})$ for some $W$ such that $\overline{W}\subset U$. $W^{2,1}_{p}(U)$ is the standard Sobolev space.
\subsection{Problem setting}
Let $(\Omega,\mathcal{F},\{\mathcal{F}_{t}\}_{t\ge 0}, \mathbb{P})$ be a filtered probability space satisfying the usual conditions and $\{B_{t}\}_{t\ge 0}$ be a standard one-dimensional Brownian motion on this space.

The state process $\{X^{\xi}_{t}\}_{t\in[0,T]}$ follows the general stochastic differential equation
\begin{displaymath}
\left\{
\begin{array}{l}
dX^{\xi}_{t}=\mu(X^{\xi}_{t},t,\xi_{t})dt+\sigma(X^{\xi}_{t},t,\xi_{t})dB_{t}-d\xi_{t},\ t\in[0,T],\\
X^{\xi}_{0-}=x_{0},
\end{array}
\right.
\end{displaymath}
where $\xi=\{\xi_{t}\}_{t\ge 0}$ is the singular control which is non-decreasing and c$\acute{a}$dl$\acute{a}$g with initial  $\xi_{0-}=0$, and the drift function $\mu$, the diffusion function $\sigma$ are exogenously given and deterministic.

In the irreversible reinsurance problem, $\{X^{\xi}_{t}\}_{t\in[0,T]}$ represents the insurance company's risk exposure controlled by reinsurance strategy $\xi$, and $\xi_{t}$ represents the accumulated amount of risk exposure covered by reinsurance up to time $t$.

The optimization problem is to find a singular control $\xi$ that minimizes the objective
\begin{align}
\label{object}
&J(x,t,y;\xi):=\mathbb{E}_{x,t,y}\bigg[\int_{t}^{T}\beta(r-t)H(X^{\xi}_{r},r,\xi_{r})dr+\beta(T-t)F(X^{\xi}_{T},\xi_{T})\\&+\int_{t}^{T}\beta(r-t)c(X^{\xi}_{r-},r,\xi_{r-})d\xi_{r}\bigg],
(x,t,y)\in\mathcal{Q}:=\mathbb{R}\times[0,T]\times(\mathbb{R_{+}}\cup\{0\}),\notag
\end{align}
where $\mathbb{E}_{x,t,y}$ denotes the expectation conditioning on $X_{t-}=x,\xi_{t-}=y$ and the last integral is defined by
\begin{align}
\label{jumpdef}
\int_{t}^{T}\beta(r-t)c(X^{\xi}_{r-},r,\xi_{r-})d\xi_{r}:=&\int_{t}^{T}\beta(r-t)c(X^{\xi}_{r},r,\xi_{r})d\xi^{c}_{r}\\&+\sum\limits_{r\in[t,T]}\beta(r-t)c(X^{\xi}_{r-},r,\xi_{r-})\Delta\xi_{r},\notag
\end{align}
with $\xi^{c}_{r}$ and $\Delta\xi_{r}:=\xi_{r}-\xi_{r-}$ being the continuous and discrete parts of $\xi_{r}$.

The objective functional (\ref{object}) takes a general form. The first term with running loss function $H$ is a penalty term penalizing high level of state process during the time horizon. The second term with terminal loss function $F$ penalizes the remaining level of state process at terminal time. The last term with cost function $c$ is the cost of controlling the state process. Future penalty and cost are discounted by the discount function $\beta$. The functions $H, F$ and $c$ are exogenously given continuous and deterministic functions bounded from below. We assume without loss of generality that the lower bounds are zeros. In addition, we assume $c\in C^{1,0,1}(\mathcal{Q}), F\in C^{2,2}(\mathbb{R}\times(\mathbb{R_{+}}\cup\{0\}))$ and that $\beta$ is in $C^{1}([0,T])$, non-increasing with $\beta(0)=1$. 

In irreversible reinsurance, the above optimization problem stands for finding a reinsurance strategy to minimize a combination of running risk loss, terminal risk loss and reinsurance cost.

Note that several successive immediate jump of $\xi$ and a lump-sum jump of $\xi$ at the same time point should result in the same objective value. Hence, at any time $t$, $\xi$ is allowed to jump at most once. In addition, the cost rate  should be assumed to keep the same during the jump, i.e.,
\begin{equation*}
c(x-a,t,y+a)=c(x,t,y), \forall(x,t,y)\in\mathcal{Q},\forall a\ge 0,
\end{equation*}
or equivalently,
\begin{equation}
\label{c-ass1}
c_{x}(x,t,y)=c_{y}(x,t,y),\forall(x,t,y)\in\mathcal{Q}.
\end{equation}

Otherwise, in the irreversible reinsurance problem, compared to a lump-sum purchase, several successive immediate reinsurance purchases at the same time point may result in different reinsurance costs.

Moreover, the optimization problem starting at $T$, which is to minimize
\begin{displaymath}
F(x-a,y+a)+c(x,T,y)a
\end{displaymath}
over $a\in[0,+\infty)$, should be well posed. We assume that the function $a\mapsto F(x-a,y+a)+c(x,T,y)a$ is strictly convex with a finite minimum point. 

\subsection{Admissible strategy}
As the discount function might not be exponential, time-inconsistency is involved. We begin by recalling the classical definition for singular control problems. Then, to deal with time-inconsistency, we introduce our novel definition of admissible strategy, or precisely \emph{admissible singular control law}. Our novel definition (Definition \ref{adml}) plays a crucial role in the next subsection where we define the equilibrium strategy.

For a singular control problem without time-inconsistency, the classical definition of an admissible singular control is as follows.

\begin{definition}
	\label{ads}
	Given initial $(x,t,y)\in\mathcal{Q}$, suppose that $\{\xi_{r}\}_{r\in[t,T]}$ is an $\{\mathcal{F}_{r}\}_{r\in[t,T]}$-adapted non-decreasing c$\acute{a}$dl$\acute{a}$g process satisfying\\
	(a) $\xi_{t-}=y$ and the stochastic differential equation
	\begin{equation*}
	\left\{
	\begin{array}{l}
dX^{\xi}_{r}=\mu(X^{\xi}_{r},r,\xi_{r})dr+\sigma(X^{\xi}_{r},r,\xi_{r})dB_{r}-d\xi_{r},\ r\in[t,T],\\
	X^{\xi}_{t-}=x.
	\end{array}
	\right.
	\end{equation*}
	has a unique strong solution $\{X^{\xi}_{r}\}_{r\in[t,T]}$.\\
	(b) The $\{\xi_{r}\}_{r\in[t,T]}$ and $\{X^{\xi}_{r}\}_{r\in[t,T]}$ given in (a) satisfy
	\begin{align*}
	&\mathbb{E}_{x,t,y}\int_{t}^{T}\beta(r-t)H(X^{\xi}_{r},r,\xi_{r})dr<\infty,\\
	&\mathbb{E}_{x,t,y}F(X^{\xi}_{T},\xi_{T})<\infty,\\
	&\mathbb{E}_{x,t,y}\int_{t}^{T}\beta(r-t)c(X^{\xi}_{r-},r,\xi_{r-})d\xi_{r}<\infty.
	\end{align*}
	Then we call $\{\xi_{r}\}_{r\in[t,T]}$ an admissible singular control on $[t,T]$. We denote the set of all admissible singular controls on $[t,T]$ by $\mathcal{D}_{[t,T]}$.
\end{definition}

Because of time-inconsistency, the above definition is not suitable. Unlike regular control, the strategy corresponding to a singular control is reflected by the increase or jump rather than the value of the control. Therefore, we need an alternative way to present the strategy. Inspired by the fact that the solution of an optimal singular control problem usually relates to a Skorohod reflection problem, we introduce the following definition.

\begin{definition}
\label{adml}
Suppose that $\Xi:=(W^{\Xi},P^{\Xi})$ is a division of $\mathcal{Q}$. $W^{\Xi}$ is an open set denoting the \emph{waiting region} and $P^{\Xi}$ is a closed set denoting the \emph{purchasing region}. We call $\Xi$ an \emph{admissible singular control law} if the following (a)$\sim$(c) hold.\\
    (a) Given any initial $(x,t,y)\in\mathcal{Q}$, the Skorohod reflection problem
	\begin{equation}
	\left\{
	\begin{array}{l}
	dX^{\xi}_{r}=\mu(X^{\xi}_{r},r,\xi_{r})dr+\sigma(X^{\xi}_{r},r,\xi_{r})dB_{r}-d\xi_{r},\ r\in[t,T],\\
	(X^{\xi}_{r},r,\xi_{r})\in\overline{W^{\Xi}}, r\in[t,T],\\
	\xi_{r}=y+\int_{t}^{r}1_{\{(X^{\xi}_{u},u,\xi_{u})\in P^{\Xi}\}}d\xi_{u},\ r\in[t,T],\\ 
	X^{\xi}_{t-}=x,\ \xi_{t-}=y.
	\end{array}
	\right.
	\label{dynamic1}
	\end{equation}
 
	has a unique strong solution $(X^{\xi},\xi):=(X^{x,t,y,\Xi},\xi^{x,t,y,\Xi})$. We call $\xi$ the \emph{singular control generated by $\Xi$ at $(x,t,y)$}.\\
	(b) The $\{\xi_{r}\}_{r\in[t,T]}$ and $\{X^{\xi}_{r}\}_{r\in[t,T]}$ given in (a) satisfy the same conditions in (b) of Definition \ref{ads}.\\
	(c) The $\{\xi_{r}\}_{r\in[t, T]}$ given in (a) has the property that, for any $\tau\in[t, T)$, as $h\rightarrow 0$,
 \begin{equation*}
\sum\limits_{r\in(\tau,\tau+h)}\Delta\xi_{r}=o(h),a.s..
\end{equation*}
\end{definition}

In Definition \ref{adml}, the properties (a) and (b) are parallel to (a) and (b) of Definition \ref{ads}. The property (c) allows $\xi$ with infinite jumps but the overall size of jumps in an infinitesimal open interval is constrained.

	The definition of admissible singular control law corresponds to the whole state-time-control space $\mathcal{Q}$. It is different from admissible control (Definition \ref{ads}) which is defined for each initial $(x,t,y)$ in $\mathcal{Q}$. However, there are some relations. The singular control generated by an admissible singular law at $(x,t,y)$ is indeed an admissible singular control on $[t,T]$. Hence, an admissible singular control law corresponds to a certain combination of admissible singular controls by varying $(x,t,y)$.

\subsection{The equilibrium singular control laws}
 In this subsection, based on the novel concept of control law introduced in Definition \ref{adml}, we extend the classical definition of equilibrium for regular control problems to singular control problems.

The well-known classical definition of equilibrium for regular control problems is introduced by \cite{ekeland2010golden}. In the classical definition, if $\hat{u}$ is the equilibrium control, then the classical definition of perturbed control $u^{h}$ is 
\begin{equation*}
u^{h}(r)=\left\{
\begin{array}{l}
u(r),\ r\in[t,t+h),\\
\hat{u}(r),\ r\in[t+h,\infty),
\end{array}
\right.
\end{equation*}
where $u$ is an arbitrary admissible control. The core step in the classical definition is to compare the objectives of the equilibrium control and the perturbed control.

However, the classical definition of perturbed control $u^{h}$ does not work for singular control problems. The classical definition replicates the values of the candidate equilibrium control $\hat{u}$ after the perturbation period $[t,t+h)$. For regular control, this means that the equilibrium strategy is followed. However, this is not the case for singular control. The main reason is that, for singular control, the pattern of increase of the control process $\xi$, rather than its values, is important. Therefore, simply replicating the values of the control process does not mean following the same strategy. Besides, the major technical obstacle with classical definition is that one cannot simply apply the It\^{o} calculus to derive an extended HJB system.

With Definition \ref{adml}, we can alternatively construct the perturbed control as follows. Consider a time horizon $[t,T]$ with $t<T$ and let $h$ be an infinitesimal time variable, and then\\
{\bf Step 1.} Impose a perturbation in $[t,t+h)$ by replacing the benchmark singular control with arbitrary admissible singular control.\\
{\bf Step 2.} At time $t+h$, let the candidate equilibrium singular control law generate the singular control on $[t+h,T]$.\\
{\bf Step 3.} Let the perturbed singular control on $[t,T]$ be the combination of  the controls on $[t,t+h)$ and $[t+h,T]$.

The constructed perturbed control for singular control problems is shown in (\ref{xih0}). Compared with the classical definition of perturbed control, the only but crucial difference appears in {\bf Step 2}. In {\bf Step 2}, we do not replicate the values of the benchmark control. Instead, we let the candidate equilibrium control law generate the control, which indeed means following the equilibrium strategy.

Another point to mention is that we need to require terminal-time optimality, i.e. condition (b) in Definition \ref{equil}, which is also the limiting case $t\uparrow T$ of the above argument. It is a special feature of our singular control model, which is different from regular control where the control at the terminal time makes no difference.

To sum up, we introduce the following definition.

\begin{definition}
	\label{equil}
	Let $\hat{\Xi}$ be an admissible singular control law. We call $\; \hat{\Xi}$ an \emph{ equilibrium singular control law} and $J(x,t,y;\xi^{x,t,y,\hat{\Xi}})$ the corresponding \emph{equilibrium value function} at $(x,t,y)$ if\\	
	(a) For any initial $(x,t,y)\in\mathcal{Q}$ with $t<T$, any $\eta\in\mathcal{D}_{[t,T]}$ with $\eta_{t-}=y$, we have
	\begin{displaymath}
	\liminf\limits_{h\rightarrow 0}\frac{J(x,t,y;\xi^{h})-J(x,t,y;\xi^{x,t,y,\hat{\Xi}})}{h}\ge 0,
	\end{displaymath}
	where $\xi^{h}, h\in(0,T-t)$, is defined by
	\begin{equation}
 \label{xih0}
	\xi^{h}_{r}=\left\{
	\begin{array}{l}
	\eta_{r}, \ r\in[t,t+h),\\
	\xi^{X^{\eta}_{(t+h)-},t+h,\eta_{(t+h)-},\hat{\Xi}}_{r}, \ r\in[t+h,T].
	\end{array}
	\right.
	\end{equation}
	(b) For any initial $(x,t,y)\in\mathcal{Q}$ with $t=T$, any $\eta\in\mathcal{D}_{[T,T]}$ with $\eta_{T-}=y$, we have
	\begin{displaymath}
	J(x,t,y;\eta)-J(x,t,y;\xi^{x,t,y,\hat{\Xi}})\ge 0.
	\end{displaymath}
\end{definition}

\begin{remark}
\label{xih}
$\xi^{h}$ in (a) is not necessarily an admissible singular control on $[t,T]$. However, $\xi^{h}$ satisfies all conditions of Definition \ref{ads} except the integrability condition (b). Actually, 
\begin{equation*}
X^{\xi_{h}}_{r}=\left\{
\begin{array}{l}
X^{\eta}_{r}, \ r\in[t,t+h),\\
X^{X^{\eta}_{(t+h)-},t+h,\eta_{(t+h)-},\hat{\Xi}}_{r}, \ r\in[t+h,T].
\end{array}
\right.
\end{equation*}
is the unique strong solution of (\ref{dynamic1}) with $\xi$ replaced by $\xi^{h}$. Hence, $J(x,t,y;\xi^{h})$ is well defined but can equal $+\infty$. As $J(x,t,y;\xi^{x,t,y,\hat{\Xi}})$ is finite by the admissibility of $\xi^{x,t,y,\hat{\Xi}}$, Definition \ref{equil} is well posed.
\end{remark}

The equilibrium singular control law defined in Definition \ref{equil} corresponds to the concept of weak equilibrium. The singular control version of the so-called strong equilibrium (see \cite{huang2021strong}) can also be defined accordingly. Moreover, sufficient and necessary conditions for strong equilibrium can be derived using the Taylor expansion approach in \cite{huang2021strong}. Limited to the scope of this paper, we omit the details. We also remark that, Definition \ref{equil} can be applied to other kinds of time-inconsistent singular control problems, not limited to non-expoential discount.

The following two sections aim to give some mathematical characterizations of the equilibrium in Definition \ref{equil}.

\begin{remark}
For simplicity, we mean ``equilibrium singular control law" when we use the term ``equilibrium".
\end{remark}

\section{Sufficient Conditions}
\label{suf}
In this section, we establish the verification theorem (Theorem \ref{verif}) which gives sufficient conditions for equilibrium.

\begin{theorem}[Verification Theorem]
	\label{verif}
	Given a function $V(x,t,y)$ and a family of functions $\{f^{s}(x,t,y)\}_{s\in[0,t]}$, define $f(x,t,y,s):=f^{s}(x,t,y)$, $(\mathcal{A}\varphi)(x,t,y):=\varphi_{t}(x,t,y)+\mu(x,t,y)\varphi_{x}(x,t,y)+\frac{1}{2}\sigma(x,t,y)^{2}\varphi_{xx}(x,t,y),\forall\varphi:\mathcal{Q}\rightarrow\mathbb{R}$ and 
	\begin{align}
	&W:=\{(x,t,y)\in \mathcal{Q}|c(x,t,y)-V_{x}(x,t,y)+V_{y}(x,t,y)>0\},\label{feed-w}\\
	&P:=\{(x,t,y)\in \mathcal{Q}|c(x,t,y)-V_{x}(x,t,y)+V_{y}(x,t,y)=0\}.\label{feed-p}
	\end{align}
	Assume the following properties:\\
	(1) 
	\begin{align*}
		&V\in C^{2,1,1}(\mathcal{Q}\cap\{(x,t,y)|t<T\})\bigcap C(\mathcal{Q}),\\
		&f^{s}\in C^{2,1,1}(\overline{W}\cap\{(x,t,y)|t<T\})\bigcap C^{1,0,1}(P\cap\{(x,t,y)|t<T\})\bigcap C(\mathcal{Q}).
	\end{align*}
	(2) $V(x,t,y)$ and $ f^{s}(x,t,y)$ satisfy
	\begin{align}
	&\min\Big\{(\mathcal{A}V)(x,t,y)+H(x,t,y)-(\mathcal{A}f)(x,t,y,t)+(\mathcal{A}f^{t})(x,t,y) \label{v1}\\
	&\qquad\qquad\qquad,c(x,t,y)-V_{x}(x,t,y)+V_{y}(x,t,y)\Big\}=0,\  \forall t<T,\notag\\
	&V(x,T,y)=\tilde{F}(x,y),\label{v4}\\
	&(\mathcal{A}f^{s})(x,t,y)+\beta(t-s)H(x,t,y)=0, \forall (x,t,y)\in \overline{W}, \ t<T,s\in[0,t],\label{v5}\\
	&\beta(t-s)c(x,t,y)-f_{x}^{s}(x,t,y)+f_{y}^{s}(x,t,y)=0,\forall (x,t,y)\in P, \ t<T,s\in[0,t],\label{v6}\\
	&f^{s}(x,T,y)=\beta(T-s)\tilde{F}(x,y).\label{v7}
	\end{align}
	for all $(x,t,y)\in\mathcal{Q}$, where
	\begin{displaymath}
	\tilde{F}(x,y):=\min\limits_{a\ge 0}\{F(x-a,y+a)+c(x,T,y)a\},\ (x,y)\in\mathbb{R}\times(\mathbb{R_{+}}\cup\{0\}).
	\end{displaymath}
	(3) $\hat{\Xi}:=(W,P)$ is an admissible singular control law.\\
	(4) For $\varphi=V,f^{s}$, $(x,t,y)\mapsto f(x,t,y,t)$ and for any $\xi\in\mathcal{D}_{[t,T]}$,
	\begin{equation}
	\label{v9}
	\mathbb{E}_{x,t,y}\int_{t}^{T}\Big[\varphi_{x}(X_{r}^{\xi},r,\xi_{r})\sigma(X_{r}^{\xi},r,\xi_{r})\Big]^{2}dr<\infty.
	\end{equation}
	Then $\hat{\Xi}=(W,P)$ is an equilibrium singular control law and $V$ is the corresponding equilibrium value function. Moreover, $f$ has the probabilistic interpretation
	\begin{align}
	\label{intf}
f(x,t,y,s)=&\mathbb{E}_{x,t,y}\bigg[\int_{t}^{T}\beta(r-s)H(X_{r}^{\hat{\xi}},r,\hat{\xi}_{r})dr+\beta(T-s)F(X_{T}^{\hat{\xi}},\hat{\xi}_{T})\\&+\int_{t}^{T}\beta(r-s)c(X_{r-}^{\hat{\xi}},r,\hat{\xi}_{r-})d\hat{\xi}_{r}\bigg],\notag
	\end{align}
	where $\hat{\xi}:=\xi^{x,t,y,\hat{\Xi}}$ is the singular control generated by $\hat{\Xi}$ at $(x,t,y)$ and $X^{\hat{\xi}}:=X^{x,t,y,\hat{\Xi}}$ is the corresponding state process.
\end{theorem}

\begin{proof} We prove Theorem \ref{verif} in three steps.
	
	{\bf Step 1.} We show that $f$ has the interpretation (\ref{intf}) and $C^{1}$ in $s$.\\
	Applying It\^{o}-Tanaka-Meyer's formula (see Sect.4 of \cite{cherny2001principal}) to $f^{s}$, we obtain for $t\in[s,T]$,
	\begin{align}
	f^{s}(X^{\hat{\xi}}_{T},T,\hat{\xi}_{T})\!-&\!f^{s}(x,t,y)=\int_{t}^{T}(\mathcal{A}f^{s})(X^{\hat{\xi}}_{r},r,\hat{\xi}_{r})dr\!+\!\int_{t}^{T}f^{s}_{x}(X^{\hat{\xi}}_{r},r,\hat{\xi}_{r})\sigma(X^{\hat{\Xi}}_{r},r,\hat{\xi}_{r})dB_{r}\notag\\
 &+\int_{t}^{T}\Big[f^{s}_{y}(X^{\hat{\xi}}_{r-},r,\hat{\xi}_{r-})\!-\!f^{s}_{x}(X^{\hat{\xi}}_{r-},r,\hat{\xi}_{r-})\Big]d\hat{\xi}^{c}_{r}\label{feq}\\
&+\sum\limits_{r\in[t,T]}\int_{0}^{\Delta\hat{\xi}_{r}}\Big[f^{s}_{y}(X^{\hat{\xi}}_{r-}\!-\!u,r,\hat{\xi}_{r-}\!+\!u)\!-\!f^{s}_{x}(X^{\hat{\xi}}_{r-}\!-\!u,r,\hat{\xi}_{r-}\!+\!u)\Big]du.\notag
	\end{align}
Thanks to (\ref{v9}), the second term on the right-hand side is a martingale.
	 As
	\begin{displaymath}
	(x,T,y)\in \overline{W}\Leftrightarrow \mathop{\arg\min}\limits_{a\ge 0}\{F(x-a,y+a)+c(x,T,y)a\}=0\Leftrightarrow \tilde{F}(x,y)=F(x,y)
	\end{displaymath}
	and $(X^{\hat{\xi}}_{T},T,\hat{\xi}_{T})\in \overline{W}$, we have $\tilde{F}(X_{T}^{\hat{\xi}},\hat{\xi}_{T})=F(X_{T}^{\hat{\xi}},\hat{\xi}_{T})$.

 Using (\ref{v5})$\sim$(\ref{v7}), (\ref{c-ass1}) and the last formula, taking expectation on both sides of (\ref{feq}) leads to (\ref{intf}).
	
	Now we show that $f$ is $C^{1}$ in $s$. As $\beta$ is $C^{1}$, the expression $\frac{\beta(r-s-\epsilon)-\beta(r-s)}{\epsilon}$ converges uniformly to $-\beta'(r-s)$ for $r\in[t,T]$ as $\epsilon\rightarrow 0$. Using (\ref{intf}), Dominated convergence theorem and (b) in Definition \ref{adml}, we have
	\begin{align*}
	f_{s}(x,t,y,s)=&\lim\limits_{\epsilon\rightarrow 0}\frac{f(x,t,y,s+\epsilon)-f(x,t,y,s)}{\epsilon}\\
	=&\mathbb{E}_{x,t,y}\bigg[-\int_{t}^{T}\beta'(r-s)H(X_{r}^{\hat{\xi}},r,\hat{\xi}_{r})dr-\beta'(T-s)F(X_{T}^{\hat{\xi}},\hat{\xi}_{T})\\
 &-\int_{t}^{T}\beta'(r-s)c(X_{r-}^{\hat{\xi}},r,\hat{\xi}_{r-})d\hat{\xi}_{r}\bigg].\end{align*}

	As $\beta$ is $C^{1}$, we obtain that $f_{s}(x,t,y,s)$ is continuous in $s$.

	{\bf Step 2.} We show that $V(x,t,y)=f(x,t,y,t)=J(x,t,y,\hat{\xi})$, and thus 
	\begin{displaymath}
		f\in C^{2,1,1,1}\Big(\big\{(x,t,y,s)\big|(x,t,y)\in\mathcal{Q},0\le s\le t<T\big\}\Big).
	\end{displaymath} 

	Similar to {\bf Step 1}, applying It\^{o}-Tanaka-Meyer's formula to $V$, taking expectations and using (\ref{v1})$\sim$(\ref{v4}), (\ref{c-ass1}), we have 
	\begin{equation}
 \label{Vexp}
V(x,t,y)\!=\!\mathbb{E}_{x,t,y}\!\bigg[F(X^{\hat{\xi}}_{T},\hat{\xi}_{T})\!-\!\int_{t}^{T}(\mathcal{A}V)(X^{\hat{\xi}}_{r},r,\hat{\xi}_{r},r)dr\!+\!\int_{t}^{T}c(X^{\hat{\xi}}_{r-},r,\hat{\xi}_{r-})d\hat{\xi}_{r}\bigg].
	\end{equation}
	Again, applying It\^{o}-Tanaka-Meyer's formula to $(x,t,y)\mapsto f(x,t,y,t)$, taking expectations and using (\ref{c-ass1}),  (\ref{v6})$\sim$(\ref{v7}), we obtain 
	\begin{equation}
 \label{fexp}
f(x,t,y,t)\!=\!\mathbb{E}_{x,t,y}\!\bigg[F(X^{\hat{\xi}}_{T},\hat{\xi}_{T})\!-\!\int_{t}^{T}\!(\mathcal{A}f)(X^{\hat{\xi}}_{r},r,\hat{\xi}_{r},r)dr\!+\!\int_{t}^{T}\!c(X^{\hat{\xi}}_{r-},r,\hat{\xi}_{r-})d\hat{\xi}_{r}\bigg].
	\end{equation}
Combination of (\ref{v1}) and (\ref{v5}) implies 
\begin{equation}
\label{av=af}
(\mathcal{A}V)(x,t,y)=(\mathcal{A}f)(x,t,y,t),\ \forall(x,t,y)\in\overline{W}.
\end{equation}
Then, from (\ref{intf}) and (\ref{Vexp})$\sim$(\ref{av=af}), we obtain
	\begin{displaymath}
	V(x,t,y)=f(x,t,y,t)=J(x,t,y,\hat{\xi}).
	\end{displaymath}
	{\bf Step 3.} We prove that $\hat{\Xi}$ is an equilibrium singular control law.
	
	{\bf Step 3.1. Proof of property (a) in Definition \ref{equil}:}
	
	For any $\xi\in\mathcal{D}_{[t,T]}$ with $\xi_{t-}=y$, define
	\begin{align}
 \label{fxi}
	f^{\xi}(x,t,y,s):=&\mathbb{E}_{x,t,y}\Bigg[\int_{t}^{T}\beta(r-s)H(X_{r}^{\xi},r,\xi_{r})dr+\beta(T-s)F(X_{T}^{\xi},\xi_{T})\\
 &+\int_{t}^{T}\beta(r-s)c(X_{r-}^{\xi},r,\xi_{r-})d\xi_{r}\Bigg].\notag
	\end{align}
	In particular, $f=f^{\hat{\xi}}$. Recalling Remark \ref{xih}, $\xi^{h}$ also corresponds to a state process $X^{\xi^{h}}$ and we can define $f^{\xi^{h}}(x,t,y,s)$ just as (\ref{fxi}).
 
 Using the fact that $\xi^{h}=\eta$ on $[t,t+h)$, it follows that
	\begin{align*}
	&J(x,t,y;\xi^{h})
=\mathbb{E}_{x,t,y}\bigg[f^{\xi^{h}}(X^{\eta}_{(t+h)-},t+h,\eta_{(t+h)-},t)+\int_{t}^{t+h}\beta(r-t)H(X^{\eta}_{r},r,\eta_{r})dr\\
&+\int_{t}^{t+h}\beta(r-t)c(X^{\eta}_{r},r,\eta_{r})d\eta^{c}_{r}+\sum\limits_{r\in[t,t+h)}\beta(r-t)c(X^{\eta}_{r-},r,\eta_{r-})\Delta\eta_{r}\bigg].
	\end{align*}
 Using the fact that $f^{\xi^{h}}(X^{\eta}_{(t+h)-},t+h,\eta_{(t+h)-},t)=f(X^{\eta}_{(t+h)-},t+h,\eta_{(t+h)-},t)$ and the last formula, we obtain
	\begin{align}
 \label{divJ}
	&J(x,t,y;\xi^{h})-J(x,t,y;\hat{\xi})\\
&\!=\!\mathbb{E}_{x,t,y}\!\bigg[\!f(X^{\eta}_{(t+h)\!-\!},t+h,\eta_{(t+h)\!-\!},t)\!-\!f(x,t,y,t)\!+\!\int_{t}^{t+h}\!\beta(r\!-\!t)H(X^{\eta}_{r},r,\eta_{r})dr\notag\\&+\int_{t}^{t+h}\beta(r-t)c(X^{\eta}_{r},r,\eta_{r})d\eta^{c}_{r}+\sum\limits_{r\in[t,t+h)}\beta(r-t)c(X^{\eta}_{r-},r,\eta_{r-})\Delta\eta_{r}\bigg].\notag
	\end{align}
	Applying It\^{o}-Tanaka-Meyer's formula to $f^{t}$ and taking expectations, we have
	\begin{align*}
	&\mathbb{E}_{x,t,y}f(X^{\eta}_{(t+h)-},t+h,\eta_{(t+h)-},t)-f(x,t,y,t)\\
	=&\mathbb{E}_{x,t,y}\bigg[\int_{t}^{t+h}(\mathcal{A}f^{t})(X^{\eta}_{r},r,\eta_{r})dr+\int_{t}^{t+h}\Big(f^{t}_{y}(X^{\eta}_{r},r,\eta_{r})-f^{t}_{x}(X^{\eta}_{r},r,\eta_{r})\Big)d\eta^{c}_{r}\\
&+\sum\limits_{r\in[t,t+h)}\int_{0}^{\Delta\eta_{r}}\Big(f^{t}_{y}(X^{\eta}_{r-}-u,r,\eta_{r-}+u)-f^{t}_{x}(X^{\eta}_{r-}-u,r,\eta_{r-}+u)\Big)du\bigg].
	\end{align*}
	Substituting the last equation into (\ref{divJ}), dividing by $h$ and letting $h\rightarrow 0$, we obtain
	\begin{align*}
	&\liminf\limits_{h\rightarrow 0}\frac{J(x,t,y;\xi^{h})-J(x,t,y;\hat{\xi})}{h}\\
	=&\liminf\limits_{h\rightarrow 0}\frac{\mathbb{E}_{x,t,y}\int_{t}^{t+h}\Big[(\mathcal{A}f^{t})(X^{\eta}_{r},r,\eta_{r})+\beta(r-t)H(X^{\eta}_{r},r,\eta_{r})\Big]dr}{h}\\
	&+\liminf\limits_{h\rightarrow 0}\frac{\mathbb{E}_{x,t,y}\int_{t}^{t+h}\Big[\beta(r-t)c(X^{\eta}_{r},r,\eta_{r})+f^{t}_{y}(X^{\eta}_{r},r,\eta_{r})-f^{t}_{x}(X^{\eta}_{r},r,\eta_{r})\Big]d\eta^{c}_{r}}{h}\\
	\ge&0,
	\end{align*}
	where the equality holds due to property (c) of Definition \ref{adml} and the inequality holds due to (\ref{v1}), (\ref{av=af}) and Fatou's lemma.
 
	{\bf Step 3.2. Proof of property (b) in Definition \ref{equil}:}
 
	We just need to prove the singular control generated by $(W,P)$ is optimal for minimizing $J(x,T,y;\xi)$. The minimization problem is equivalent to finding optimal jump $a=\Delta\xi_{T}\ge 0$ to minimize
	\begin{displaymath}
	a\mapsto F(x-a,y+a)+c(x,T,y)a.
	\end{displaymath}
	By the first order condition, the optimal strategy at $T$ is to jump (purchase a positive amount of reinsurance) if and only if 
	\begin{displaymath}
	c(x,T,y)-F_{x}(x,y)+F_{y}(x,y)<0.
	\end{displaymath}
By (\ref{feed-w}), we have
 \begin{align*}
 (x,T,y)\in W\Leftrightarrow&c(x,T,y)-V_{x}(x,T,y)+V_{y}(x,T,y)>0\\\Leftrightarrow&c(x,T,y)-\tilde{F}_{x}(x,y)+\tilde{F}_{y}(x,y)>0.
 \end{align*}
 Hence, proving property (b) of Definition \ref{equil} boils down to showing 
 \begin{equation}
 \label{cond-b}
 c(x,T,y)-F_{x}(x,y)+F_{y}(x,y)>0\Leftrightarrow c(x,T,y)-\tilde{F}_{x}(x,y)+\tilde{F}_{y}(x,y)>0.
 \end{equation}
 
 On the one hand, if $c(x,T,y)-F_{x}(x,y)+F_{y}(x,y)>0$, then $a=0$ minimizes $a\mapsto F(x-a,y+a)+c(x,T,y)a$. By definition, $\tilde{F}(x,y)=F(x,y)$ and we obtain the $\Rightarrow$ direction of (\ref{cond-b}).
 
 On the other hand, if $c(x,T,y)-F_{x}(x,y)+F_{y}(x,y)\le 0$, then, by definition, we have $\tilde{F}(x,y)=F(x-a(x,y),y+a(x,y))+c(x,T,y)a(x,y)$ with $a(x,y)$ satisfying the first order condition
 \begin{equation}
 \label{first}
 c(x,T,y)-F_{x}(x-a(x,y),y+a(x,y))+F_{y}(x-a(x,y),y+a(x,y))=0.
 \end{equation}
 The implicit function theorem implies that $a(x,y)$ is $C^{1}$.
 Hence, using (\ref{c-ass1}) and (\ref{first}), we have 
\begin{align*}
&\tilde{F}_{x}(x,y)-\tilde{F}_{y}(x,y)\\
=&(1-a_{x}(x,y))F_{x}(x-a(x,y),y+a(x,y))+a_{x}(x,y)F_{y}(x-a(x,y),y+a(x,y))\\
&+c_{x}(x,T,y)a(x,y)+c(x,T,y)a_{x}(x,y)
+a_{y}(x,y)F_{x}(x-a(x,y),y+a(x,y))\\&-(1+a_{y}(x,y))F_{y}(x-a(x,y),y+a(x,y))-c_{y}(x,T,y)a(x,y)-c(x,T,y)a_{y}(x,y)\Big]\\
=&F_{x}(x-a(x,y),y+a(x,y))-F_{y}(x-a(x,y),y+a(x,y)).
\end{align*}
Then, with (\ref{first}), we obtain
\begin{align*}
&c(x,T,y)-\tilde{F}_{x}(x,y)+\tilde{F}_{y}(x,y)\\
=&c(x,T,y)-F_{x}(x-a(x,y),y+a(x,y))+F_{y}(x-a(x,y),y+a(x,y))\\
=&0.
\end{align*}
and the proof of (\ref{cond-b}) is completed.

Combining {\bf Step 3.1} and {\bf Step 3.2}, $\hat{\Xi}$ is an equilibrium singular control.
\end{proof}

\begin{remark}
The extended HJB system (\ref{v1})$\sim$(\ref{v7}) reduces to the classical HJB system for the time-consistent singular control problem when the discount function is exponential, i.e., $\beta(t)=e^{-\gamma t}$. Specifically, by the probabilistic interpretation (\ref{intf}), we have $f(x,t,y,s)=e^{-\gamma(t-s)}V(x,t,y)$. Substituting it into (\ref{v1})$\sim$(\ref{v7}), we obtain 
\begin{equation}
\label{classicalhjb}
\left\{
\begin{array}{l}
\min\{\gamma V(x,t,y)+(\mathcal{A}V)(x,t,y),c(x,t,y)-V_{x}(x,t,y)+V_{y}(x,t,y)\}=0,\forall t<T,\\
V(x,T,y)=\tilde{F}(x,y),
\end{array}
\right.
\end{equation}
which is exactly the classical HJB system for the time-consistent singular control problem. A point to mention is that (\ref{classicalhjb}) corresponds to our {\bf Problem 1} in Subsection \ref{subproblem1} , which consists of (\ref{v1})$\sim$(\ref{v4}), and our {\bf Problem 2} in Subsection \ref{subproblem2}, which consists of (\ref{v5})$\sim$(\ref{v7}), is solved by $f(x,t,y,s)=e^{-\gamma(t-s)}V(x,t,y)$ naturally if $V(x,t,y)=e^{\gamma(t-s)}f(x,t,y,s)$ solves {\bf Problem 1}.
\end{remark}

To sum up, the core sufficient condition for equilibrium is to satisfy the extended HJB equations given in (2) of Theorem \ref{verif}. We will show in the next section that it is nearly necessary at the same time.
\section{Necessary Conditions}
\label{nec}

In this section, we show that the extended HJB equations are also necessary under some regularity conditions. Specifying that the extended HJB equations rely on $W$ and $P$, we further show that generally $W$ and $P$ are exactly the feedback forms given in the verification theorem.

\begin{theorem}
	\label{necessary}
	Suppose that $\hat{\Xi}=(W,P)$ is an equilibrium singular control law with equilibrium value function $V(x,t,y)$. Let $\hat{\xi}:=\xi^{x,t,y,\hat{\Xi}}$ and $ X^{\hat{\xi}}:=X^{x,t,y,\hat{\Xi}}$ be the singular control generated by $\hat{\Xi}$ at $(x,t,y)$ and the corresponding risk exposure process respectively. Define $f^{s}(x,t,y):=f(x,t,y,s)$ with $f(x,t,y,s)$ given by (\ref{intf}). Assuming\\
	(i)
		\begin{displaymath}
	f\in C^{2,1,1,1}\Big(\big\{(x,t,y,s)\big|(x,t,y)\in\mathcal{Q},0\le s\le t<T\big\}\Big).
	\end{displaymath} 
	(ii) For any $(x,t,y)\in\mathcal{Q}$, the singular control $\tilde{\eta}:=\{\tilde{\eta}_{r}\equiv y,\forall r\in[t,T]\}$ is admissible.\\
	(iii) For any $(x,t,y)\in\mathcal{Q}$ with $t<T$, any $s\in[0,t]$ and $\eta=\tilde{\eta}$ or $\hat{\xi}$, there exists $\epsilon_{0}:=\epsilon_{0}(x,t,y,s,\eta)$ such that
	\begin{align}
&\mathbb{E}_{x,t,y}\int_{t}^{t+\epsilon_{0}}\Big[f^{s}_{x}(X^{\eta}_{r},r,\eta_{r})\sigma(X^{\eta}_{r},r,\eta_{r})\Big]^{2}dr<\infty,\label{as1}\\
&\mathbb{E}_{x,t,y}\sup_{h\in(0,\epsilon_{0})}\bigg|\frac{\int_{t}^{t+h}\Big[(Af^{s})(X^{\eta}_{r},r,\eta_{r})+\beta(r-s)H(X^{\eta}_{r},r,\eta_{r})\Big]dr}{h}\bigg|<\infty.\label{as2}
	\end{align}
	Then the extended HJB system (2) in Theorem \ref{verif} holds.
\end{theorem}

\begin{proof}
	{\bf Step 1.} We prove the terminal conditions (\ref{v4}) and (\ref{v7}).
	
	By definition of $V$, we have 
	\begin{displaymath}
	V(x,T,y)=F(x+\Delta\hat{\xi}_{T},y-\Delta\hat{\xi}_{T})+c(x,T,y)\Delta\hat{\xi}_{T}.
	\end{displaymath}
	By (b) of Definition \ref{equil}, $\Delta\hat{\xi}$ optimizes the right-hand side of the last equation. Hence, (\ref{v4}) holds.
	
	By definition of $f$, we have
	\begin{align*}
	f(x,T,y,s)=&\beta(T-s)\Big[F(x+\Delta\hat{\xi}_{T},y-\Delta\hat{\xi}_{T})+c(x,T,y)\Delta\hat{\xi}_{T}\Big]\\
	=&\beta(T-s)V(x,T,y).
	\end{align*}
	Then (\ref{v7}) follows.
	
	{\bf Step 2.} We prove (\ref{v5}) and (\ref{v6}).
	Let us fix arbitrary $(x,t,y)\in\mathcal{Q}$ with $t<T$ and $s\in[0,t]$. Let $h\in(0,\min\{\epsilon_{0},T-t\})$.

By (\ref{intf}), we have
\begin{align*}
&\mathbb{E}_{x,t,y}f^{s}(X^{\hat{\xi}}_{(t+h)-},t+h,\hat{\xi}_{(t+h)-})-f^{s}(x,t,y)\\
=&\mathbb{E}_{x,t,y}\bigg[-\int_{t}^{t+h}\beta(r-s)H(X^{\hat{\xi}}_{r},r,\hat{\xi}_{r})dr-\int_{t}^{t+h}\beta(r-s)c(X^{\hat{\xi}}_{r},r,\hat{\xi}_{r})d\hat{\xi}^{c}_{r}\\
&-\sum\limits_{r\in[t,t+h)}\beta(r-s)c(X^{\hat{\xi}}_{r-},r,\hat{\xi}_{r-})\Delta\hat{\xi}_{r}\bigg].
\end{align*}

Applying It\^{o}-Tanaka-Meyer's formula and taking expectations, we have
\begin{align*}
&\mathbb{E}_{x,t,y}f^{s}(X^{\hat{\xi}}_{(t+h)-},t+h,\hat{\xi}_{(t+h)-})-f^{s}(x,t,y)\\
=&\mathbb{E}_{x,t,y}\bigg[\int_{t}^{t+h}(\mathcal{A}f^{s})(X^{\hat{\xi}}_{r},r,\hat{\xi}_{r})dr+\int_{t}^{t+h}\Big(f^{s}_{y}(X^{\hat{\xi}}_{r},r,\hat{\xi}_{r})-f^{s}_{x}(X^{\hat{\xi}}_{r},r,\hat{\xi}_{r})\Big)d\hat{\xi}^{c}_{r}\\
&+\sum\limits_{r\in[t,t+h)}\int_{0}^{\Delta\hat{\xi}_{r}}\Big(f^{s}_{y}(X^{\hat{\xi}}_{r-}-u,r,\hat{\xi}_{r-}+u)-f^{s}_{x}(X^{\hat{\xi}}_{r-}-u,r,\hat{\xi}_{r-}+u)\Big)du\bigg],
\end{align*}
where the expectation of It\^{o} integral vanishes due to (\ref{as1}).

Combining the last two equations and using (\ref{c-ass1}) lead to
\begin{align}
\label{fs}
0=&\mathbb{E}_{x,t,y}\bigg[\sum\limits_{r\in[t,t+h)}\int_{0}^{\Delta\hat{\xi}_{r}}\Big(f^{s}_{y}(X^{\hat{\xi}}_{r-}-u,r,\hat{\xi}_{r-}+u)-f^{s}_{x}(X^{\hat{\xi}}_{r-}-u,r,\hat{\xi}_{r-}+u)\\
&\qquad\qquad\qquad\quad-\!\beta(r\!-\!s)c(X^{\hat{\xi}}_{r-}\!-\!u,r,\hat{\xi}_{r-}\!+\!u)\Big)du\notag\\
&\!+\!\int_{t}^{t+h}\Big(f^{s}_{y}(X^{\hat{\xi}}_{r},r,\hat{\xi}_{r})\!-\!f^{s}_{x}(X^{\hat{\xi}}_{r},r,\hat{\xi}_{r})\!+\!\beta(r\!-\!s)c(X^{\hat{\xi}}_{r},r,\hat{\xi}_{r})\Big)d\hat{\xi}^{c}_{r}\notag\\
&\!+\!\int_{t}^{t+h}\Big((\mathcal{A}f^{s})(X^{\hat{\xi}}_{r},r,\hat{\xi}_{r})\!+\!\beta(r-s)H(X^{\hat{\xi}}_{r},r,\hat{\xi}_{r})\Big)dr\bigg].\notag
\end{align}
Letting $h\downarrow 0$, we obtain 
\begin{equation*}
\mathbb{E}_{x,t,y}\Big[f^{s}(x-\Delta\hat{\xi}_{t},t,y+\Delta\hat{\xi}_{t})+\beta(t-s)c(x,t,y)\Delta\hat{\xi}_{t}\Big]=f^{s}(x,t,y).
\end{equation*}
As $\Delta\hat{\xi}_{t}=\min\{a\ge 0|(x-a,t,y+a)\in\overline{W}\},a.s.$ for $(x,t,y)\in P$, we obtain (\ref{v6}).

By Definition \ref{adml}, for $(x,t,y)$ in the interior of $W$, there exists $h_{1}$ small enough such that $\hat{\xi}$ is flat in $[t,t+h_{1}]$ almost surely and thus the first two terms on the right-hand side of (\ref{fs}) vanishes as $h\downarrow 0$. Dividing both sides of (\ref{fs}) by $h$ and letting $h\downarrow 0$, we obtain
\begin{displaymath}
\lim\limits_{h\downarrow 0}\mathbb{E}_{x,t,y}\frac{\int_{t}^{t+h}\Big[(\mathcal{A}f^{s})(X^{\hat{\xi}}_{r},r,\hat{\xi}_{r})+\beta(r-s)H(X^{\hat{\xi}}_{r},r,\hat{\xi}_{r})\Big]dr}{h}=0.
\end{displaymath}
In light of (\ref{as2}), using Dominated convergence theorem, we obtain
\begin{displaymath}
	(\mathcal{A}f^{s})(x,t,y)+\beta(t-s)H(x,t,y)=0, \forall (x,t,y)\in int(W), \forall s\in[0,t].
\end{displaymath}
Then (\ref{v5}) holds by (i) and the continuity of $H$.

{\bf Step 3.} Finally, we prove (\ref{v1}).

In light of (\ref{v5}), (\ref{v6}) and $V(x,t,y)=f(x,t,y,t)$, we just need to further prove the inequalities
\begin{align}
&c(x,t,y)-V_{x}(x,t,y)+V_{y}(x,t,y)\ge 0,\label{r1}\\
&(\mathcal{A}f^{t})(x,t,y)+H(x,t,y)\ge 0.\label{r2}
\end{align}

Using (\ref{as1}) and the techniques in {\bf Step 3.1} in the proof of Theorem \ref{verif}, we have, for any $(x,t,y)\in\mathcal{Q}$ with $t<T$, $\forall h\in(0,\min\{\epsilon_{0},T-t\})$, $\forall\eta\in\mathcal{D}_{[t,T]}$,
\begin{align}
\label{J-J}
&J(x,t,y;\xi^{h})-J(x,t,y;\hat{\xi})\\
=&\mathbb{E}_{x,t,y}\bigg[\int_{t}^{t+h}\Big((\mathcal{A}f^{t})(X^{\eta}_{r},r,\eta_{r})+\beta(r-t)H(X^{\eta}_{r},r,\eta_{r})\Big)dr\notag\\
&+\sum\limits_{r\in[t,t+h)}\int_{0}^{\Delta\eta_{r}}\Big(f^{t}_{y}(X^{\eta}_{r-}-u,r,\eta_{r-}+u)-f^{t}_{x}(X^{\eta}_{r-}-u,r,\eta_{r-}+u)\notag\\
&\qquad\qquad\qquad\qquad+\beta(r-t)c(X^{\eta}_{r-}-u,r,\eta_{r-}+u)\Big)du\notag\\
&+\int_{t}^{t+h}\Big(\beta(r-t)c(X^{\eta}_{r},r,\eta_{r})+f^{t}_{y}(X^{\eta}_{r},r,\eta_{r})-f^{t}_{x}(X^{\eta}_{r},r,\eta_{r})\Big)d\eta^{c}_{r}\bigg].\notag
\end{align}
By definition of equilibrium, we have $\liminf\limits_{h\downarrow 0}(J(x,t,y;\xi^{h})-J(x,t,y;\hat{\xi}))\ge 0$, which yields
\begin{displaymath}
\int_{0}^{\Delta\eta_{t}}\Big[c(x-u,t,y+u)+f^{t}_{y}(x-u,t,y+u)-f^{t}_{x}(x-u,t,y+u)\Big]du\ge 0.
\end{displaymath}
As $\Delta\eta_{t}\ge 0$ can be arbitrary, we obtain (\ref{r1}).

Moreover, $\liminf\limits_{h\downarrow 0}\frac{J(x,t,y;\xi^{h})-J(x,t,y;\hat{\xi})}{h}\ge 0$. Choosing $\eta_{r}\equiv\eta_{t-},\forall r\in[t,T]$ and applying Dominated convergence theorem, we have (\ref{r2}).

Thus, the proof is complete.
\end{proof}

 As a special feature of singular control problems, the extended HJB equations are associated with $W$ and $P$, which are constructed by (\ref{feed-w})$\sim$(\ref{feed-p}) in the verification theorem. The next proposition indicates that, generally, the $W$ and $P$ in Theorem \ref{necessary} exactly have the forms given by (\ref{feed-w})$\sim$(\ref{feed-p}).

\begin{proposition}
	\label{necuni}
	Under all conditions (i)$\sim$(iii) of Theorem \ref{necessary}, suppose that\\
	 (iv) For any $\mathcal{S}\subset\mathcal{Q}$ with nonempty interior, there is no $g\in C^{2,1,1}(\mathcal{S})$ simultaneously satisfying
	\begin{align*}
    &(\mathcal{A}g)(x,t,y)=-H(x,t,y),\\
    &g_{y}(x,t,y)-g_{x}(x,t,y)=-c(x,t,y)
	\end{align*}
	on $\mathcal{S}$.
	
	 Then the equilibrium singular control law $(W,P)$ in Theorem \ref{necessary} exactly has the form given by (\ref{feed-w})$\sim$(\ref{feed-p}).
\end{proposition}

\begin{proof}
	According to the proof of Theorem \ref{necessary}, we have
	\begin{align*}
	&W\subset\{(x,t,y)\in\mathcal{Q}|(\mathcal{A}f^{t})(x,t,y)+H(x,t,y)=0\},\\
	&P\subset\{(x,t,y)\in\mathcal{Q}|c(x,t,y)-V_{x}(x,t,y)+V_{y}(x,t,y)=0\},\\
	&W\cup P=\mathcal{Q}.
	\end{align*}
	If $P\neq \{(x,t,y)\in\mathcal{Q}|c(x,t,y)-V_{x}(x,t,y)+V_{y}(x,t,y)=0\}$, let $\mathcal{S}:=\{(x,t,y)\in\mathcal{Q}|c(x,t,y)-V_{x}(x,t,y)+V_{y}(x,t,y)=0\}\setminus P$. By property (a) of Definition \ref{adml}, $P$ is a closed set. Hence, $\mathcal{S}$ has a nonempty interior. Let $g(x,t,y):=f^{t}(x,t,y)=V(x,t,y)$. Then $g\in C^{2,1,1}(\mathcal{S})$ and satisfies
	\begin{align*}
	&(\mathcal{A}g)(x,t,y)=-H(x,t,y),\\
	&g_{y}(x,t,y)-g_{x}(x,t,y)=-c(x,t,y)
	\end{align*}
	on $\mathcal{S}$, leading to a contradiction. Thus, Proposition \ref{necuni} follows.

\end{proof} 

\begin{remark}
	The condition (iv) in Proposition \ref{necuni} is quite weak. For instance, if all $\mu,\sigma,F,H,c$ are independent of $y$ (the same as in the next section), then (iv) holds whenever the expression $H(x,t)+\mu(x,t)c(t)+c'(t)x$ has a dependence on $x$.
\end{remark}

\section{On the Existence of Equilibria}
\label{exi}

As illustrated in Sections \ref{suf}-\ref{nec}, studying the system of HJB equations (\ref{v1})$\sim$(\ref{v7}) is crucial to studying the equilibrium. With some additional assumptions, we show the existence of solutions and thus the existence of equilibria.

The existence of equilibria for the problem starting at $T$ is trivial. In the following, we focus on problems starting at $t<T$.

We only consider the case where $y$ is not involved. That is, we assume that all $\mu,\sigma,F,H,c$ are independent of $y$. Nevertheless, this assumption is common in real problems such as irreversible reinsurance and dividend problems. Thus, with (\ref{c-ass1}), the cost function is only a function of $t$. And our assumed convexity property becomes the assumption that $F$ is convex and $a\mapsto F(x-a)+c(T)a$ has finite minimum point. Now, (\ref{v1})$\sim$(\ref{v7}) become
\begin{equation}
\label{HJB}
\left\{
\begin{aligned}
&\min\Big\{(\mathcal{A}V)(x,t)+H(x,t)-f_{s}(x,t,t),c(t)-V_{x}(x,t)\Big\}=0,\forall t<T,\\
&V(x,T)=\tilde{F}(x),\forall x\in\mathbb{R},\\
&(\mathcal{A}f^{s})(x,t)+\beta(t-s)H(x,t)=0, \forall (x,t)\in \overline{W}, t<T,\ s\in[0,t],\\
&\beta(t-s)c(t)-f_{x}^{s}(x,t)=0,\forall (x,t)\in P, t<T, s\in[0,t],\\
&f^{s}(x,T)=\beta(T-s)\tilde{F}(x),\forall x\in\mathbb{R},
\end{aligned}
\right.
\end{equation}
for $(x,t)\in\mathbb{R}\times[0,T)$ with
\begin{align*}
&W:=\{(x,t)\in\mathbb{R}\times[0,T)|c(t)-V_{x}(x,t)>0\},\\
&P:=\{(x,t)\in\mathbb{R}\times[0,T)|c(t)-V_{x}(x,t)=0\}.
\end{align*}
We introduce the following notations
\begin{align*}
&Q_{T-s}=\mathbb{R}\times(0,T-s],\ Q^{-}_{T-s}=(-\infty,0]\times(0,T-s],\\
&Q^{n}_{T-s}=[-n,n]\times(0,T-s],\ Q^{n,-}_{T-s}=[-n,0]\times(0,T-s].
\end{align*}

For an existence result, we additionally need the following assumptions:
\begin{assumption}
	\label{all}
	 We assume that all partial derivatives of $\mu,\sigma,H,\tilde{F}$ appearing in the followings exist almost everywhere (abbr. a.e.) and there exist $3<p<\infty$, $0<\alpha<1-\frac{3}{p}$ and $\kappa:(-\infty,-1]\times[0, T]\rightarrow[0,+\infty)$ such that\\
	 (a)(regularity) $\kappa(\cdot,t)$ is in $L^{1}((-\infty,-1])$ for any $t\in[0,T]$. And for any $n\ge 1$,
	 \begin{equation*}
	 \sigma_{x}(x,T-t),\mu_{x}(x,T-t),H_{x}(x,T-t),c(T-t)\in C^{\alpha,\frac{\alpha}{2}}(\overline{Q^{n}_{T}}),H_{xx}(x,T-t)\in L^{\infty}(Q^{n}_{T}).
	 \end{equation*}
	 (b)(asymptotic properties)
	\begin{align*}
	&\lim\limits_{x\rightarrow-\infty}H(x,T-t)=0,\lim\limits_{x\rightarrow-\infty}H_{xx}(x,T-t)=0,\forall t\in(0,T],\\
	&\lim\limits_{x\rightarrow-\infty}\mu_{xx}(x,T-t)\kappa(x,T-t)=0,\forall t\in[0,T],\\
 &\lim\limits_{x\rightarrow-\infty}F(x)=0,\lim\limits_{x\rightarrow+\infty}F'(x)=+\infty.
	\end{align*}
	(c)(inequalities) There exists $M>0$ such that for any $(x,t)\in\overline{Q_{T}}$,
	\begin{align}
	&\sigma(x,T-t)>0, \mu_{x}(x,T-t)\le 0, \mu_{xx}(x,T-t)\ge 0, \mu_{xxx}(x,T-t)\ge 0,\label{c134}\\
	&0\le-\beta'(0)\tilde{F}''(x)\le H_{xx}(x,T-t),\label{c2}\\
 &c'(T-t)\le 0, \sigma_{t}(x,T-t)\le 0, \mu_{xt}(x,T-t)\le 0,\label{t1}\\
 &\sigma(x,T-t)\sigma_{xt}(x,T-t)+\sigma_{t}(x,T-t)\sigma_{x}(x,T-t)+\mu_{t}(x,T-t)\le 0,\label{t2}\\
	&\sigma(x,T-t)\sigma_{xx}(x,T-t)+\sigma_{x}^{2}(x,T-t)+2\mu_{x}(x,T-t)\le 0,\label{c5}\\
 \label{c6}
	&\max\limits_{\tau\in[0,T-t]}\bigg\{-\frac{\beta'(\tau)}{\beta(\tau)}\bigg\}c(T-t)+c'(T-t)+\mu_{x}(x,T-t)c(T-t)\\&\le -H_{x}(x,T-t)<\beta'(0)\tilde{F}'(x)\le 0,\ \mathcal{L}\tilde{F}'(x)\ge 0, a.e.,\notag\\
  \label{c7}
 	&\mathcal{L}\big(H_{x}(x,T-t)\big)\le\beta'(0)H_{x}(x,T-t), a.e.,\\
  \label{c8}
	&\frac{\partial}{\partial x}\Big[\mathcal{L}\big(H_{x}(x,T-t)\big)\Big]\le\beta'(0)H_{xx}(x,T-t),a.e.,\\
&M\mathcal{L}\big(\kappa(x,T-t)\big)\le -H_{x}(x,T-t),a.e. x\le -1,\label{c9}\\
&\tilde{F}'(x)\le M\kappa(x,T-t),\forall x\le -1,\label{c10}
\end{align}
 where the operator $\mathcal{L}$ is defined by (\ref{mL}) below.
\end{assumption}
\begin{remark}
\label{example}
Assumption \ref{all} only relies on the {\it data} of the problem, and can be verified in specific problems. For example, in irreversible reinsurance, if the risk exposure is mean-reverting ($\mu(x,t)=b-ax$) with constant volatility $\sigma$, the reinsurance cost $c$ is constant, the terminal loss is exponential ($F(x)=C_{F}e^{\Psi_{F}x}$), and the running loss $H$ is time-dependent exponential ($H(x,t)=C_{H}e^{\Psi_{H}(t)x}$) for $x$ below some $x_{0}$ and grows linearly above $x_{0}$. Assumption \ref{all} holds for $\kappa(x,t)=e^{C_{\kappa}x e^{\kappa t}}$ if the following conditions hold successively. First, $b,\sigma,C_{F},C_{H}>0$ and $\Psi_{F}>\max\{1,\frac{1}{b}\},a\in(\max\{b\!+\!\frac{1}{2}\sigma^{2},1\},b\Psi_{F}\!+\!\frac{1}{2}\sigma^{2}\Psi_{F}^{2}),\kappa>a,C_{\kappa}\in(0,\Psi_{F}e^{-(\kappa+a)T})$. Second, $\Psi_{H}(t)$ satisfies
\begin{equation*}
\min\limits_{t\in[0,T]}\Psi_{H}(t)>C_{\kappa}e^{\kappa T},\ \max\limits_{t\in[0,T]}\Psi_{H}(t)<\Psi_{F},\ \min\limits_{t\in[0,T]}\frac{\Psi_{H}'(t)}{\Psi_{H}(t)}>a.
\end{equation*}
For example, the $\Psi_{H}(t)$ might be exponential or linear subject to some constraints. Third, set $\kappa\in(\max\limits_{t\in[0,T]}\Psi_{H}(t),\Psi_{F})$ and $0<c< C_{F}\Psi_{F}e^{\Psi_{F}K}$ where
\begin{equation*}
K=\min\limits_{t\in[0,T]}\frac{-2\frac{\Psi_{H}'(t)}{\Psi_{H}(t)}-\frac{1}{2}\sigma^{2}\Psi_{H}^{2}(t)-b\Psi_{H}(t)+2a}{\Psi_{H}'(t)-a\Psi_{H}(t)}.
\end{equation*}
Fourth, the constant $x_{0}$ satisfies
\begin{equation*}
\frac{1}{\Psi_{F}}\ln(\frac{c}{C_{F}\Psi_{F}})<x_{0}<\min\{\frac{1}{\Psi_{F}}\ln(\frac{ac}{C_{F}\Psi_{F}}),K\}.
\end{equation*}
Finally, the discount $\beta(t)$ satisfies the following conditions on $-\beta'(0)$ and $\frac{\beta'(t)}{\beta(t)}$:
\begin{align*}
&-\beta'(0)\le\min\limits_{t\in[0,T]}\bigg\{\frac{C_{H}\Psi_{H}^{2}(t)}{c\Psi_{F}}(\frac{c}{C_{F}\Psi_{F}})^{\frac{\Psi_{H}(t)}{\Psi_{F}}}\bigg\},\\
&-\beta'(0)\le\min\limits_{t\in[0,T]}\bigg\{-2\frac{\Psi_{H}'(t)}{\Psi_{H}(t)}-(\Psi_{H}'(t)-a\Psi_{H}(t))x_{0}-\frac{1}{2}\sigma^{2}\Psi_{H}^{2}(t)-b\Psi_{H}(t)+2a\bigg\},\\
&\min\limits_{t\in[0,T]}\frac{\beta'(t)}{\beta(t)}\ge \max\limits_{t\in[0,T]}\bigg\{\frac{C_{H}\Psi_{H}(t)}{c}e^{\Psi_{H}(t)x_{0}}\bigg\}-a.
\end{align*}
The inequalities with operator $\mathcal{L}$ are kinds of parabolic inequalities, which are common conditions for deriving needed estimations (e.g., see \cite{ladyzhenskaya1988estimates} and \cite{ladyzhenskaya1985estimates} ). In stochastic control theory, these inequalities can be viewed as conditions for some processes being supermartingales. For example, if $\mu(x,t)=b-ax$ and $\sigma(x,t)=\sigma$, then (\ref{c7}) is equivalent to the process $e^{[-a-\beta'(0)]t}H_{x}(X_{t},t)$ being a supermartingale, and (\ref{c8}) corresponds to the process $e^{[\sigma^2-2a-\beta'(0)]t}H_{xx}(X_{t},t)$. In this case, (\ref{c7}) and (\ref{c8}) require that the marginal running loss and its margin have the tendency to decrease after properly discounted.
\end{remark}

In the following subsections, we first successively study the existence of solutions to two sub-problems and finally prove the existence of solutions to (\ref{HJB}).

\subsection{First sub-problem}\label{subproblem1}
In this subsection, we focus on the existence and some weakened uniqueness of solutions to the following sub-problem:

{\bf Problem 1.} Given $f$, solve the terminal value problem for $V$:
\begin{equation*}
\left\{
\begin{aligned}
&\min\Big\{(\mathcal{A}V)(x,t)+H(x,t)-f_{s}(x,t,t),c(t)-V_{x}(x,t)\Big\}=0,\ \forall (x,t)\in\mathbb{R}\times[0,T),\\
&V(x,T)=\tilde{F}(x),\ \forall x\in\mathbb{R}.
\end{aligned}
\right.
\end{equation*}

First, we introduce the following set that will be useful later.
\begin{align*}
\mathcal{R}_{0}:=\{f(x,t,s)|&f(x,T-t,\cdot)\in C^{1}([0,T-t]),\forall (x,t)\in Q_{T},\\
&f_{s}(x,T-t,T-t)\in C^{1,0}(\overline{Q_{T}}), f_{sx}(x,T-t,T-t)\in C^{\alpha,\frac{\alpha}{2}}(Q_{T}),\\
&0\le f_{sx}(x,T-t,T-t)< H_{x}(x,T-t),\forall (x,t)\in \overline{Q_{T}},\\
&\lim\limits_{x\rightarrow-\infty}f_{s}(x,T-t,T-t)=0,\forall t\in(0,T],\\
&0\le f_{sxx}(x,T-t,T-t)\le H_{xx}(x,T-t),\forall (x,t)\in Q_{T}\}.
\end{align*}
We show the existence of solutions to {\bf Problem 1} with $f$ given in $\mathcal{R}_{0}$.

Define $g(x,t)=V(x,T-t)$. From {\bf Problem 1}, we have a problem of $g$:
\begin{equation}
\label{eqg}
\left\{
\begin{aligned}
&\!\min\!\Big\{\!\mathcal{\tilde{A}}g(x,t)\!+\!H(x,T\!-\!t)\!-\!f_{s}(x,T\!-\!t,T\!-\!t), c(T\!-\!t)\!-\!g_{x}(x,t)\!\Big\}\!=\!0, (x,t)\!\in\! Q_{T},\\
&g(x,0)=\tilde{F}(x), x\in\mathbb{R}.
\end{aligned}
\right.
\end{equation}
where $\mathcal{\tilde{A}}g(x,t):=-g_{t}(x,t)+\mu(x,T-t)g_{x}(x,t)+\frac{1}{2}\sigma^{2}(x,T-t)g_{xx}(x,t)$. From (b) of Assumption \ref{all}, we have another boundary condition $g(-\infty,t)=0, t\in(0,T]$.

Using differentiability of $g,\mu,\sigma,H,F,f_s$ and some prior conjectures, the existence of solutions to Problem (\ref{eqg}) can be derived from the existence of solutions to the following problem for $v(x,t)=g_{x}(x,t)$:
\begin{equation}
\label{eqv}
\left\{
\begin{aligned}
&\!\min\!\Big\{\!\mathcal{L}v(x,t)\!+\!H_{x}(x,T\!-\!t)\!-\!f_{sx}(x,T\!-\!t,T\!-\!t),c(T\!-\!t)\!-\!v(x,t)\!\Big\}\!=\!0, (x,t)\!\in\! Q_{T},\\
&v(x,0)=\tilde{F}'(x), x\in\mathbb{R},
\end{aligned}
\right.
\end{equation}
where the operator $\mathcal{L}$ is defined by
\begin{align}
\label{mL}
\mathcal{L}v(x,t):=&\!-v_{t}(x,t)\!+\!\frac{1}{2}\sigma^{2}(x,T\!-\!t)v_{xx}(x,t)\\
&\!+\!\big[\sigma(x,T\!-\!t)\sigma_x(x,T\!-\!t)\!+\!\mu(x,T\!-\!t)\big]v_{x}(x,t)\!+\!\mu_{x}(x,T\!-\!t)v(x,t).\notag
\end{align}

The problem (\ref{eqv}) is an obstacle problem with unbounded domain. We consider the following approximation problem with $n$ sufficiently large (such that $c(T-t)=\tilde{F}'(n)$) and $f\in\mathcal{R}_{0}$:
\begin{equation}
\label{vpb}
\left\{
\begin{aligned}
&\!\mathcal{L}\!v^{\epsilon,n}(x,t)\!+\!H_{x}\!(x,T\!-\!t)\!-\!f_{sx}\!(x,T\!-\!t,T\!-\!t)\!+\!\alpha_{\epsilon,n}\!(c\!(T\!-\!t)\!-\!v^{\epsilon,n}\!(x,t))\!=\!0, (\!x\!,\!t\!)\!\in\! Q^{n}_{T},\\
&v^{\epsilon,n}(-n,t)=\tilde{F}'(-n),\ v^{\epsilon,n}(n,t)=c(T-t), t\in(0,T],\\
&v^{\epsilon,n}(x,0)=\tilde{F}'(x), x\in[-n,n].
\end{aligned}
\right.
\end{equation}
The penalty function $\alpha_{\epsilon,n}(\cdot)=C_{n}\alpha_{\epsilon}(\cdot)\in C^{\infty}(\mathbb{R}),\epsilon>0$, satisfies
\begin{align*}
&\alpha_{\epsilon}(0)=-1,\ \alpha_{\epsilon}(z)\le 0,\ \alpha'_{\epsilon}(z)\ge 0,\ \alpha''_{\epsilon}(z)\le 0,\ \forall z\in\mathbb{R},\\
&\lim\limits_{\epsilon\rightarrow 0}\alpha_{\epsilon}(z)=0,\ \forall z>0;\ \lim\limits_{\epsilon\rightarrow 0}\alpha_{\epsilon}(z)=-\infty,\ \forall z<0,
\end{align*}
 and $C_{n}$ are positive constants satisfying $C_{n}<\min\limits_{(t,x)\in \overline{Q^{n}_{T-s}}}[H_{x}\!(x,T\!-\!t)\!-\!f_{sx}\!(x,T\!-\!t,T\!-\!t)].$

\begin{theorem}
	\label{thpb}
	Under Assumption \ref{all}, for any $f\in\mathcal{R}_{0}$, Problem (\ref{vpb}) has a  $W^{2,1}_{p}(Q^{n}_{T})$ $\bigcap C^{2+\alpha,1+\frac{\alpha}{2}}(Q^{n}_{T})$-solution $v^{\epsilon,n}$. Moreover, for $n$ sufficiently large, 
	\begin{align}
	&\tilde{F}'(x)\le v^{\epsilon,n}(x,t)\le c(T-t),\forall (x,t)\in Q^{n}_{T},\label{v-esti}\\
 &v^{\epsilon,n}(x,t)\le C\kappa(x,T-t), \forall (x,t)\in[-n,-1]\times[0,T],\label{v-esti2}\\
	&0\le v^{\epsilon,n}_{x}(x,t), 0\le v^{\epsilon,n}_{t}(x,t),\forall (x,t)\in Q^{n}_{T},\label{vx-esti}
	\end{align}
	where $C$ is some constant independent of $n$.
\end{theorem}
\begin{proof}
	{\bf Step 1.}
	For all $N>0$, we define $\alpha_{\epsilon,n,N}(z)=\max\{\alpha_{\epsilon,n}(z),-N\},\ \forall z\in\mathbb{R}$. Then there exists a $W^{2,1}_{p}(Q^{n}_{T})\bigcap C^{2+\alpha,1+\frac{\alpha}{2}}(Q^{n}_{T})$-solution $v^{\epsilon,n,N}$ to (\ref{vpb}) with $\alpha_{\epsilon,n}$ replaced by $\alpha_{\epsilon,n,N}$. The proof of this step is a standard application of $L^{p}$ theory of the initial boundary value problem (Theorem 1.1 of \cite{wang2021nonlinear}) and Schauder's fixed point theorem (Theorem 3 of \cite{evans2022partial} Sect. 9.2.2). We omit the details.
	
	{\bf Step 2.} We prove that for $N$ sufficiently large, $v^{\epsilon,n,N}$ becomes independent of $N$ and is a $W^{2,1}_{p}(Q^{n}_{T})\bigcap C^{2+\alpha,1+\frac{\alpha}{2}}(Q^{n}_{T})$-solution to (\ref{vpb}).
	
	Define $L(x,t)=\alpha_{\epsilon,n,N}(c(T-t)-v^{\epsilon,n,N}(x,t)),\ \forall(x,t)\in Q^{n}_{T}$, and let $(x_0,t_0)$ be a minimum point of $L(x,t)$ in $Q^{n}_{T}$. Suppose $L(x_0,t_0)<-C_{n}$, then by an argument similar to the proof of maximum principle, we have
	\begin{align*}
	-L(x_0,t_0)=&\mathcal{L}v^{\epsilon,n,N}(x_0,t_0)+H_{x}(x_0,T-t_0)-f_{sx}(x_0,T-t_0,T-t_0)\\
	\le&\mathcal{L}c(T-t_0)+H_{x}(x_0,T-t_0)-f_{sx}(x_0,T-t_0,T-t_0),
	\end{align*}
	where the right-hand side is non-positive by (\ref{c6}), contradicting the hypothesis. Hence, for $N$ larger than $C_{n}$, $v^{\epsilon,n}:=v^{\epsilon,n,N}$ is a $C^{2+\alpha,1+\frac{\alpha}{2}}(Q^{n}_{T})$ solution to (\ref{vpb}).
	
	{\bf Step 3.} Properties (\ref{v-esti})$\sim$(\ref{vx-esti}) follows directly by using the comparison principle and Assumption \ref{all}.
\end{proof}

\begin{theorem}
	\label{exist-v}
	Under Assumption \ref{all}, for any $f\in\mathcal{R}_{0}$, there exists a \\$W^{2,1}_{p}(Q_{T})\bigcap C^{1,0}(\overline{Q_{T}})$-solution $v$ to Problem (\ref{eqv}) satisfying
	\begin{align*}
	&0\le v(x,t)\le C\kappa(x,T-t),\ \forall (x,t)\in (-\infty,-1]\times[0,T],\\
	&v_{x}(x,t)\ge 0, v_{t}(x,t)\ge 0,\ \forall (x,t)\in Q_{T}.
	\end{align*}
\end{theorem}
\begin{proof}
	Set $Q^{(m)}=[a^{(m)},b^{(m)}]\times(t^{(m)},T],\forall m\ge 1$ with $a^{(m)}\downarrow-\infty,b^{(m)}\uparrow+\infty,t^{(m)}\downarrow 0$.
	For any $m$, set $N^{(m)}>\max\{-a^{(m)},b^{(m)}\}$. By interior $L^{p}$ estimate (Theorem 1.9 of \cite{wang2021nonlinear}), we have, for $n\ge N^{(m+1)}$,
	\begin{align*}
	\Vert v^{\epsilon,n}\Vert_{W^{2,1}_{p}(Q^{(m)})}
	\le\tilde{C}\Big[&\Vert H_{x}(x,T-t)\Vert_{L^{p}(Q^{(m+1)})}+\Vert f_{sx}(x,T-t,T-t)\Vert_{L^{p}(Q^{(m+1)})}\\
 &+1+\Vert c(T-t)\Vert_{L^{p}(Q^{(m+1)})}\Big],
	\end{align*}
	where $\tilde{C}$ is independent of $\epsilon$ and $n$. By Embedding theorem, $\Vert v^{\epsilon,n}\Vert_{C^{1+\alpha,\frac{1+\alpha}{2}}(\overline{Q^{(m)}})}$ is also bounded by constant independent of $\epsilon$ and $n$.
	Hence, for $m=1$, there exists a sequence $v^{\epsilon^{(1)}_{k},n^{(1)}_{k}},k\ge 1$ with $\epsilon^{(1)}_{k}\downarrow 0, N^{(2)}\le n^{(1)}_{k}\uparrow\infty$ satisfying $v^{\epsilon^{(1)}_{k},n^{(1)}_{k}}\rightarrow v^{(1)}$  weakly in $W^{2,1}_{p}(Q^{(1)})$, strongly in $C^{1+\alpha,\frac{1+\alpha}{2}}(\overline{Q^{(1)}})$ for some $v^{(1)}$.
	For $m=2$, there exists a subsequence of $v^{\epsilon^{(1)}_{k},n^{(1)}_{k}},k\ge 1$ denoted by $v^{\epsilon^{(2)}_{k},n^{(2)}_{k}},k\ge 1$ with $\epsilon^{(2)}_{k}\downarrow 0, N^{(3)}\le n^{(2)}_{k}\uparrow\infty$ satisfying $v^{\epsilon^{(2)}_{k},n^{(2)}_{k}}\rightarrow v^{(2)}, k\rightarrow\infty$ weakly in $W^{2,1}_{p}(Q^{(2)})$, strongly in $C^{1+\alpha,\frac{1+\alpha}{2}}(\overline{Q^{(2)}})$ for some $v^{(2)}$. We can repeat the above process and find $v^{\epsilon^{(m)}_{k},n^{(m)}_{k}},k\ge 1$ and $v^{(m)}$ for each $m\ge 1$ satisfying the corresponding properties.
	 By the above proof, we have that $v^{(m)},m\ge 1$ are consistent. Hence, we can define $v(x,t)=v^{(m)}(x,t), \forall(x,t)\in Q^{(m)},m\ge 1$. Then $v\in W^{2,1}_{p}(Q_{T})\bigcap C^{1,0}(\overline{Q_{T}})$.

	By a standard proof in the penalty method, $v$ solves (\ref{eqv}). Hence, we conclude that $v$ is a $W^{2,1}_{p}(Q_{T})\bigcap C^{1,0}(\overline{Q_{T}})$ solution to (\ref{eqv}). The inequalities are obvious in light of Theorem \ref{thpb}.
\end{proof}

		Now we can define the waiting-purchasing boundary $\Gamma$ given by $v$ in Theorem \ref{exist-v}. As $v_{x}(x,t)\ge 0$, the function $x\mapsto c(T-t)-v(x,t)$ is non-increasing. Hence, for any fixed $t\in[0,T]$, there exists
\begin{displaymath}
\Gamma(t):=\inf\{x|c(T-t)-v(x,t)=0\}.
\end{displaymath}

\begin{theorem}
	\label{exist-g}
Under Assumption \ref{all}, for any $f\in\mathcal{R}_{0}$, Problem (\ref{eqg}) has a $C^{2,1}(Q_{T})$-solution $g(x,t):=\int_{-\infty}^{x}v(z,t)dz$, where $v$ is a solution to (\ref{eqv}) given in Theorem \ref{exist-v}.
\end{theorem}
\begin{proof}
	$g$ is well-defined as $0\le v(x,t)\le C\kappa(x,T-t)$.
 
	As a byproduct, we have $v(x,t)<c(T-t)$ when $x$ is sufficiently small. By differentiating (\ref{eqv}) and It{\^o}'s formula, for x sufficiently small,
	\begin{align*}
	v_{x}(t,x)=&\mathbb{E}_{x,T-t}\bigg\{\tilde{F}''(X^{v}_{T})\!+\!\int_{T-t}^{T}\Big[H_{xx}\!(X^{v}_{r},r)\!-\!f_{sxx}\!(X^{v}_{r},r,r)\!+\!\mu_{xx}\!(X^{v}_{r},r)v\!(X^{v}_{r},T\!-\!r)\\&+(\sigma(X^{v}_{r},r)\sigma_{xx}(X^{v}_{r},r)+\sigma_{x}^{2}(X^{v}_{r},r)+2\mu_{x}(X^{v}_{r},r))v_{x}(X^{v}_{r},T-r)\Big]dr\bigg\},
	\end{align*}
	where $X^{v}$ follows the dynamic
	\begin{equation}
 \label{Xv}
	\left\{
	\begin{aligned}
	&dX^{v}_{r}=(2\sigma(X^{v}_{r},r)\sigma_{x}(X^{v}_{r},r)+\mu(X^{v}_{r},r))dr+\sigma(X^{v}_{r},r)dB_{r},\ r\in[T-t,T],\\
	&X^{v}_{t}=x.
	\end{aligned}
	\right.
	\end{equation}
Hence,
\begin{displaymath}
0\le v_{x}(t,x)\le \mathbb{E}_{x,T-t}\bigg\{\tilde{F}''(X^{v}_{T})+\int_{T-t}^{T}\Big[H_{xx}(X^{v}_{r},r)+\mu_{xx}(X^{v}_{r},r)v(X^{v}_{r},T-r)\Big]dr\bigg\}.
\end{displaymath}
Note that the expectation term corresponds to an objective with terminal cost $x\mapsto\tilde{F}''(x)$ and running cost $(x,t)\mapsto H_{xx}(x,t)+\mu_{xx}(x,t)v(x, T-t)$. Then, by (b) of Assumption \ref{all}, we have $\lim\limits_{x\rightarrow-\infty}v_{x}(x,t)=0$. Hence, we obtain $\lim\limits_{x\rightarrow-\infty}g(x,t)\!=\!0,\ \lim\limits_{x\rightarrow-\infty}g_{x}(x,t)\!$ $=\!0,\ \lim\limits_{x\rightarrow-\infty}g_{xx}(x,t)\!=\!0$ and $v,v_x,v_{xx}$ are all in $L^{1}((-\infty,0])$. As $f\in\mathcal{R}_{0}$, we have $H_{x}(x,T-t)-f_{sx}(x,T-t,T-t)\in L^{1}((-\infty,0])$. As such, we have $v_{t}\in L^{1}((-\infty,0])$. Thus, $g$ is differentiable w.r.t $t$ with $g_{t}(x,t)=\int_{-\infty}^{x}v_{t}(z,t)dz$. Similar to the proof of Corollary 4.2 in \cite{friedman1975parabolic}, we have that $v_t$ is continuous and thus $g\in C^{2,1}(Q_{T})$ with
		\begin{equation*}
	\lim\limits_{x\rightarrow-\infty}g(x,t)=0,\ \lim\limits_{x\rightarrow-\infty}g_{x}(x,t)=0,\ \lim\limits_{x\rightarrow-\infty}g_{xx}(x,t)=0,\ \lim\limits_{x\rightarrow-\infty}g_{t}(x,t)=0.
	\end{equation*}
	
Therefore, we can integrate (\ref{eqv}) from $x=-\infty$ and obtain that $g$ solves (\ref{eqg}).
\end{proof}

Then we have the existence result for {\bf Problem 1}:
\begin{corollary}
	\label{exist-p1}
	Under Assumption \ref{all}, for any $f\in\mathcal{R}_{0}$, {\bf Problem 1} has a classical solution $V(x,t)=g(x,T-t)$.
\end{corollary}

Now we focus on the uniqueness of $v$. First, we introduce some useful properties of the waiting purchasing boundary $\Gamma$.

\begin{proposition}
	\label{gamma}
	The waiting-purchasing boundary $\Gamma(\cdot)$ takes finite values and is in $C^{0}\big([0,T]\big)$.
\end{proposition}
\begin{proof}
	From $v(x,t)\le C\kappa(x,T-t)$, we have $\Gamma(t)>-\infty$. 
	
    By Theorem \ref{exist-g}, Corollary \ref{exist-p1}, we have $\Gamma(t)=\inf\big\{x\big|c(T-t)-V_{x}(x,T-t)=0\big\}$, where $V$ is the value function of a time-consistent singular control problem with discount $t\mapsto 1$, running loss $(x,t)\mapsto H(x,t)-f_{s}(x,t,t)$ and terminal loss $x\mapsto F(x)$. By our assumption that $a\mapsto F(x-a)+c(t)a$ is convex with a finite minimum point, we obtain that $\xi$ jumps immediately for sufficiently large initial $x$. Hence, $\Gamma(t)<+\infty$.
	
	As $v\in C^{1,0}(\overline{Q_{T}})$, we have $\Gamma\in C^{0}\big([0,T]\big)$.
\end{proof}

\begin{theorem}
	\label{uniq-v}
	In Theorem \ref{exist-v}, the solution $v\in W^{2,1}_{p}(Q_{T})\bigcap C^{1,0}(\overline{Q_{T}})$ satisfying
	\begin{align*}
	&0\le v(x,t)\le C\kappa(x,T-t),\ \forall (x,t)\in Q_{T},\\
	&v_{x}(x,t)\ge 0,\ \forall (x,t)\in Q_{T}
	\end{align*}
	is unique.
\end{theorem}
\begin{proof}
	Suppose that two different $v_{1},v_{2}\in W^{2,1}_{p}(Q_{T})\bigcap C^{1,0}(\overline{Q_{T}})$ solves (\ref{eqv}) with properties
	\begin{align*}
	&0\le v_{i}(x,t)\le C\kappa(x,T-t),\ \forall (x,t)\in Q_{T},i=1,2,\\
	&v_{ix}(x,t)\ge 0,\ \forall (x,t)\in Q_{T},i=1,2.
	\end{align*}
	We assume without loss of generality that there exists $(x_0,t_0)$ such that $v_{1}(x_0,t_0)>v_{2}(x_0,t_0)$. Set $0<\epsilon_{0}<v_{1}(x_0,t_0)-v_{2}(x_0,t_0)$. Let $S_0$ be the biggest connected open set containing $(x_0,t_0)$ on which $v_1-v_2>\epsilon_{0}$ holds.
 By $0\le v_{i}(x,t)\le C\kappa(x,T-t),\ \forall (x,t)\in Q_{T},i=1,2$, and Proposition \ref{gamma}, for $i=1,2$,
	\begin{displaymath}
	\lim\limits_{x\rightarrow-\infty}v_{i}(x,t)=0, \lim\limits_{x\rightarrow+\infty}v_{i}(x,t)=c(T-t),\forall t\in[0,T].
	\end{displaymath}
	Hence, $v_{1}=v_{2}+\epsilon_{0}$ on $\partial S_{0}\bigcap\{(x,t)|t<T\}$.

	Using the fact that $\mathcal{L}v_{2}(x,t)=-H_{x}(x, T-t)+f_{sx}(x, T-t, T-t)\le \mathcal{L}v_{1}(x,t)$, by an argument similar to the maximum principle, we have $v_{1}\le v_{2}+\epsilon_{0}$ on $S_{0}$, contradicting $v_{1}(x_0,t_0)>v_{2}(x_0,t_0)+\epsilon_{0}$. Thus, the uniqueness holds.
\end{proof}

\subsection{Second sub-problem}\label{subproblem2}
In this subsection, we study the existence of solutions to the following sub-problem:

{\bf Problem 2:} For each $s\in[0,T)$, given $V(x,t)$ as the solution of {\bf Problem 1}, define
\begin{align*}
&W:=\{(x,t)\in\mathbb{R}\times[s,T)|c(t)-V_{x}(x,t)>0\},\\
&P:=\{(x,t)\in\mathbb{R}\times[s,T)|c(t)-V_{x}(x,t)=0\},
\end{align*}
and solve the following terminal value problem for $f^{s}(x,t)$:
\begin{equation*}
	\left\{
	\begin{aligned}
		&(\mathcal{A}f^{s})(x,t)+\beta(t-s)H(x,t)=0, \forall (x,t)\in \overline{W},\\
		&\beta(t-s)c(t)-f_{x}^{s}(x,t)=0,\forall (x,t)\in P,\\
		&f^{s}(x,T)=\beta(T-s)\tilde{F}(x),\forall x\in\mathbb{R}.
	\end{aligned}
	\right.
\end{equation*}

By Theorem \ref{exist-v} and Theorem \ref{uniq-v}, we can define the mapping $\mathcal{Y}_{1}$ that maps any $f\in\mathcal{R}_{0}$ to the unique solution $v$ of (\ref{eqv}). Moreover, we define 
\begin{align*}
\mathcal{R}_{1}=\big\{v(x,t)\big|&v(x,t)\in W^{2,1}_{p}(Q_{T})\bigcap C^{1,0}(\overline{Q_{T}}),\\
&0\le v(x,t)\le C\kappa(x,T-t),\ \forall (x,t)\in (-\infty,-1]\times[0,T],\\
&v(x,t)\le c(T-t),\ \forall (x,t)\in Q_{T},\\
	&v_{x}(x,t)\ge 0,\ \forall (x,t)\in Q_{T}\big\}.
\end{align*}
Then $\mathcal{Y}_{1}(\mathcal{R}_{0})\subset\mathcal{R}_{1}$.
Similarly to {\bf Problem 1}, define $g^{s}(x,t)=f^{s}(x,T-t)$, and then the problem for $g^{s}$ is as follows. \\ Given $v\in\mathcal{R}_{1}$,
\begin{equation}
\label{eqgs}
\left\{
\begin{aligned}
&(\mathcal{\tilde{A}}g^{s})(x,t)+\beta(T-s-t)H(x,T-t)=0, \forall (x,t)\in \overline{W_{T-s}},\\
&\beta(T-t-s)c(T-t)-g_{x}^{s}(x,t)=0,\forall (x,t)\in P_{T-s},\\
&g^{s}(x,0)=\beta(T-s)\tilde{F}(x),\ \forall x\in\mathbb{R},
\end{aligned}
\right.
\end{equation}
where
\begin{align*}
&W_{T-s}:=\{(x,t)\in\mathbb{R}\times(0,T-s]|x<\Gamma(t)\},\\
&P_{T-s}:=\{(x,t)\in\mathbb{R}\times(0,T-s]|x\ge\Gamma(t)\}.
\end{align*}
and $\Gamma$ is determined by $\Gamma(t)=\inf\{x|c(T-t)=v(x,t)\}$.

We solve Problem (\ref{eqgs}) separately in $\overline{W_{T-s}}$ and $P_{T-s}$. By the smoothness of $f$ in Theorem \ref{verif}, we impose the following boundary condition:
\begin{displaymath}
g^{s}_{x}(\Gamma(t),t)=\beta(T-t-s)c(T-t), \forall t\in(0,T-s].
\end{displaymath}

First, we solve Problem (\ref{eqgs}) in $\overline{W_{T-s}}$, define $h^{s}(x,t)=g^{s}(x+\Gamma(t),t)$, and then the equations for $h^{s}$ are
\begin{equation}
\label{eqh}
\left\{
\begin{aligned}
&\bar{\mathcal{A}}h^{s}(x,t)+\beta(T-s-t)H(x+\Gamma(t),T-t)=0,\ \forall(x,t)\in Q^{-}_{T-s},\\
&h^{s}(x,0)=\beta(T-s)\tilde{F}(x+\Gamma(t)),\ \forall x\in(-\infty,0],\\
&h^{s}_{x}(0,t)=\beta(T-t-s)c(T-t),\ \forall t\in(0,T-s],
\end{aligned}
\right.
\end{equation}
where 
\begin{equation*}
\bar{\mathcal{A}}h^{s}(x,t):=-h^{s}_{t}(x,t)+[\Gamma'(t)+\mu(x+\Gamma(t),T-t)]h^{s}_{x}(x,t)+\frac{1}{2}\sigma^{2}(x+\Gamma(t),T-t)h^{s}_{xx}(x,t).
\end{equation*} 
We may assume that $\Gamma$ is $C^{1}$, otherwise we can use a sequence of $C^{1}$ function $\Gamma^{n}$ to approach $\Gamma$. The equations for $k^{s}(x,t):=h^{s}_{x}(x,t)$ are
\begin{equation}
\label{eqk}
\left\{
\begin{aligned}
&\bar{\mathcal{L}}k^{s}(x,t)+\beta(T-s-t)H_{x}(x+\Gamma(t),T-t)=0,\ \forall(x,t)\in Q^{-}_{T-s},\\
&k^{s}(x,0)=\beta(T-s)\tilde{F}'(x+\Gamma(t)),\ \forall x\in(-\infty,0],\\
&k^{s}(0,t)=\beta(T-t-s)c(T-t),\ \forall t\in(0,T-s],
\end{aligned}
\right.
\end{equation}
where
\begin{align*}
\bar{\mathcal{L}}k^{s}\!(x,t)\!:=&\!-\!k^{s}_{t}\!(x,t)\\&\!+\![\mu\!(x\!+\!\Gamma\!(t),T\!-\!t)\!+\!\sigma\!(x\!+\!\Gamma\!(t),T\!-\!t)\sigma_{x}\!(x\!+\!\Gamma\!(t),T\!-\!t)+\Gamma'\!(t)]k^{s}_{x}\!(x,t)\\&+\frac{1}{2}\sigma^{2}(x+\Gamma(t),T-t)k^{s}_{xx}(x,t)+\mu_{x}(x+\Gamma(t),T-t)k^{s}(x,t).
\end{align*} 
The approximation problem for (\ref{eqk}) is 
\begin{equation}
\label{bk}
\left\{
\begin{aligned}
&\bar{\mathcal{L}}k^{s,n}(x,t)+\beta(T-s-t)H_{x}(x+\Gamma(t),T-t)=0,\ \forall(x,t)\in Q^{n,-}_{T-s},\\
&k^{s,n}\!(0,t)\!=\!\beta\!(T\!-\!t\!-\!s)c\!(T\!-\!t),\ k^{s,n}\!(\!-\!n,t)\!=\!\beta\!(T\!-\!s)\tilde{F}'\!(\!-\!n\!+\!\Gamma\!(t)),\ \forall t\!\in\!(0,T\!-\!s],\\
&k^{s,n}(x,0)=\beta(T-s)\tilde{F}'(x+\Gamma(0)),\ \forall x\in[-n,0].
\end{aligned}
\right.
\end{equation}

\begin{theorem}
Under Assumption \ref{all}, for any $v\in\mathcal{R}_{1}$, there exists a unique $W^{2,1}_{p}(Q^{n,-}_{T-s})\bigcap C^{2+\alpha,1+\frac{\alpha}{2}}$ $(Q^{n,-}_{T-s})$-solution $k^{s,n}$ to Problem (\ref{bk}) and it holds that, for $n$ sufficiently large,
\begin{align}
&\beta(T\!-\!s)\tilde{F}'(-\!n\!+\!\Gamma(t))\le k^{s,n}(x,t)
\le \beta(T\!-\!t\!-\!s)c(T\!-\!t),\forall (x,t)\in Q_{T\!-\!s}^{n,-},\label{k1}\\
&k^{s,n}(x,t)\le C\kappa(x\!+\!\Gamma(t),T\!-\!t),\forall t\in[0,T\!-\!s],x\in[-\!n,0\wedge(-\!1-\!\Gamma(t))],\label{k2}\\
&0\le k^{s,n}_{x}(x,t),\forall (x,t)\in Q^{n,-}_{T-s},\label{k3}
\end{align}
where $C$ is some constant independent of $n$.
\end{theorem}
\begin{proof}
	Existence and uniqueness follow by the $L^p$ theory for the first initial-boundary value problem. We only need to prove the inequalities (\ref{k1})$\sim$(\ref{k3}).
	
	Let $l(x,t)=k^{s,n}(x-\Gamma(t),t),\forall t\in (0,T-s],x\in [-n+\Gamma(t),\Gamma(t)]$, then we have
	\begin{equation*}
	\left\{
	\begin{aligned}
	&\mathcal{L}l(x,t)
	+\beta(T-s-t)H_{x}(x,T-t)=0, t\in (0,T-s],x\in [-n+\Gamma(t),\Gamma(t)],\\
	&l(\!-\!n\!+\!\Gamma\!(t),t)\!=\!\beta\!(T\!-\!s)\tilde{F}'\!(\!-\!n\!+\!\Gamma\!(t)),\ l(\Gamma\!(t),t)\!=\!\beta\!(T\!-\!t\!-\!s)c\!(T\!-\!t), t\!\in\!(0,T\!-\!s],\\
	&l(x,0)=\beta(T-s)\tilde{F}'(x), x\in[-n+\Gamma(0),\Gamma(0)].
	\end{aligned}
	\right.
	\end{equation*}
	Thus, for $n$ sufficiently large, $\mathcal{L}l(x,t)\le \mathcal{L}(\beta(T-s)\tilde{F}'(-n+\Gamma(t)))$ and we have $l(x,t)\ge \beta(T-s)\tilde{F}'(-n+\Gamma(t))$ by the comparison principle.
	By (\ref{c6}) of Assumption \ref{all}, $\mathcal{L}l(x,t)\ge\mathcal{L}(\beta(T-t-s)c(T-t))$ and we obtain $l(x,t)\le \beta(T-t-s)c(T-t)$ using the comparison principle again. Now we have $l_{x}(x,t)\ge 0$ at the boundaries, by (\ref{c134}) and (\ref{c5}) of Assumption \ref{all} and the comparison principle, we have $l_{x}(x,t)\ge 0$.
	
 	Define $C=\max\{M,\max\limits_{t\in [0,T-s]}c(T-t)\}$. Then, for $n$ sufficiently large, $\mathcal{L}l(x,t)\ge\mathcal{L}(C\kappa(x,T-t))$
 and we obtain $l(x,t)\le C\kappa(x,T-t)$ for $t\in[0,T-s],x\in[-n+\Gamma(t),-1\wedge\Gamma(t)]$.
 
 Thus, using the transformation $k^{s,n}(x,t)=l(x+\Gamma(t),t)$, we obtain all the required inequalities.
\end{proof}
\begin{theorem}
	\label{exist-k}
	Under Assumption \ref{all}, for any $v\in\mathcal{R}_{1}$, Problem (\ref{eqk}) has a unique solution $k^{s}\in C^{2+\alpha,1+\frac{\alpha}{2}}(Q^{-}_{T-s})\bigcap C^{1,0}(\overline{Q^{-}_{T-s}})$ satisfying
	\begin{align*}
&0\le k^{s}(x,t)
\le \beta(T-t-s)c(T-t),\forall (x,t)\in Q^{-}_{T-s},\\
&k^{s}(x,t)\le C\kappa(x+\Gamma(t),T-t),\forall t\in[0,T-s],x\in[-\infty,0\wedge(-1-\Gamma(t))],\\
&0\le k^{s}_{x}(x,t),\forall (x,t)\in Q^{-}_{T-s}.
\end{align*}
\end{theorem}
\begin{proof}
	Similar to the proof of Theorem \ref{exist-v} and Theorem \ref{uniq-v}, we omit it here.
\end{proof}
\begin{theorem}
	\label{exist-h}
	Under Assumption \ref{all}, for any $v\in\mathcal{R}_{1}$, Problem (\ref{eqh}) has a $C^{3,1}(Q^{-}_{T-s})$ $\bigcap C^{2,0}(\overline{Q^{-}_{T-s}})$-solution $h^{s}(x,t)=\int_{-\infty}^{x}k^{s}(z,t)dz$, where $k^{s}$ is a solution to Problem (\ref{eqk}) given in Theorem \ref{exist-k}.
\end{theorem}
\begin{proof}
	The proof is similar to the proof of Theorem \ref{exist-g}, and we omit it here.
\end{proof}

\begin{theorem}
	\label{exist-f}
	Under Assumption \ref{all}, for any $v\in\mathcal{R}_{1}$, {\bf Problem 2} has a solution $f^{s}(x,t)$ given by
	\begin{align*}
	&f^{s}(x,t)=g^{s}(x,T-t),\forall (x,t)\in Q_{T-s},\\
	&g^{s}(x,t)=\left\{
	\begin{aligned}
	&h^{s}(x-\Gamma(t),t),\forall (x,t)\in\overline{W_{T-s}},\\
	&h^{s}(0,t)+\beta(T-s-t)c(T-t)(x-\Gamma(t)),\forall(x,t)\in P_{T-s},
	\end{aligned}
	\right.
	\end{align*}
	where $h^{s}$ is given in Theorem \ref{exist-h}. Moreover, $f(x,t,s):=f^{s}(x,t)$ is in $\mathcal{R}_{0}$.
\end{theorem}
\begin{proof}
	It is trivial that $f^{s}(x,t)$ solves {\bf Problem 2}. We only prove $f\in\mathcal{R}_{0}$.
	
	According to the proof of Theorem \ref{verif} and $k^{s}\in C^{2+\alpha,1+\frac{\alpha}{2}}(Q^{-}_{T-s})\bigcap C^{1,0}(\overline{Q^{-}_{T-s}})$, we have 
	\begin{equation*}
	f(x,T-t,\cdot)\in C^{1}([0,T-t]),\forall (x,t)\in Q_{T},
	\end{equation*}
	and
	\begin{displaymath}
	f_{s}(x,t,s)=\mathbb{E}_{x,t}\bigg[-\int_{t}^{T}\beta'(r-s)H(X_{r}^{\hat{\xi}},r)dr-\beta'(T-s)F(X_{T}^{\hat{\xi}})-\int_{t}^{T}\beta'(r-s)c(r)d\hat{\xi}_{r}\bigg].
	\end{displaymath}
	It follows that 
	\begin{displaymath}
	\lim\limits_{x\rightarrow-\infty}f_{s}(x,T-t,T-t)=0,\forall t\in(0,T-s].
	\end{displaymath}
	Note that $u(x,t):=f_{s}(x,t,t)$ satisfies
\begin{equation*}
\left\{
\begin{aligned}
&\mathcal{A}u(x,t)-\beta'(0)H(x,t)=0,\ \forall (x,t)\in \overline{W},\\
&-\beta'(0)c(t)-u_{x}(x,t)=0,\ \forall (x,t)\in P,\\
&u(x,T)=-\beta'(0)\tilde{F}(x).
\end{aligned}
\right.
\end{equation*}
Then $w(x,t):=u_{x}(x,T-t)=f_{sx}(x,T-t,T-t)$ satisfies
\begin{equation}
\label{eq-w}
\left\{
\begin{aligned}
&\mathcal{L}w(x,t)-\beta'(0)H_{x}(x,T-t)=0,\ \forall (x,t)\in \overline{W_{T-s}},\\
&-\beta'(0)c(T-t)-w(x,t)=0,\ \forall (x,t)\in P_{T-s},\\
&w(x,0)=-\beta'(0)\tilde{F}'(x).
\end{aligned}
\right.
\end{equation}
Applying the comparison principle to the equations of $k^{s}_{s}$, we have that $k^{s}_{s}$ is also controlled by $\kappa(x+\Gamma(t),T-t)$ (multiplied by some constant). Then 
\begin{equation*}
f_{s}(x,t,s)=\int_{-\infty}^{x-\Gamma(T-t)}k^{s}_{s}(z,T-t)dz
\end{equation*}
and we have $\lim\limits_{x\rightarrow-\infty}w(x,t)=0$.

Based on (\ref{c7}) of Assumption \ref{all} and the maximum principle, we have $0\le w(x,t)<H_{x}(x,T-t)$. Then $0\le f_{sx}(x,T-t,T-t)<H_{x}(x,T-t)$ for all $(x,t)\in \overline{Q_{T-s}}$.

Differentiating the equation for $f_{sx}$ and using  It{\^o}'s formula, we have, for $x\le\Gamma(t)$,
\begin{align*}
\!f_{sxx}\!(x\!,\!T\!-\!t\!,\!T\!-\!t)\!=\!\mathbb{E}_{x\!,\!T\!-\!t}\!\bigg\{\!-\!\beta'\!(\!0\!)\tilde{F}''\!(\!X^{v}_{T}\!)\!+\!\int_{T\!-\!t}^{T}\!\Big[\!-\!\beta'\!(\!0\!)H_{xx}\!(\!X^{v}_{r}\!,\!r\!)\!+\!\mu_{xx}\!(\!X^{v}_{r}\!,\!r\!)w(\!X^{v}_{r}\!,\!T\!-\!r\!)\!&\\\!+\!(\sigma(X^{v}_{r},r)\sigma_{xx}(X^{v}_{r},r)\!+\!\sigma_{x}^{2}(X^{v}_{r},r)\!+\!2\mu_{x}(X^{v}_{r},r))f_{sxx}(X^{v}_{r},T\!-\!r,T\!-\!r)\Big]dr\bigg\}.&
\end{align*}
	where $X^{v}$ follows (\ref{Xv}).
By the boundary condition, we have $f_{sxx}(\Gamma(t),T\!-\!t,T\!-\!t)=0$.

We claim that $f_{sxx}(x,T-t,T-t)\ge 0$ in $Q_{T-s}$. Otherwise, for each $t\in(0,T-s]$, there exists
\begin{equation*}
x^{t}_{0}=\inf\{x\le \Gamma(t)|f_{sxx}(x,T-t,T-t)\ge 0,\forall x\in(x^{t}_{0},\Gamma(t)]\}>-\infty.
\end{equation*}
Then $f_{sxx}(x^{t}_{0},T-t,T-t)=0$. On the other hand, using (\ref{c2}),(\ref{c134}) and (\ref{c5}) of Assumption \ref{all}, we have a contradiction that $f_{sxx}(x^{t}_{0},T-t,T-t)>0$.

 Hence, $f_{sxx}(x,T-t,T-t)\ge 0$ for all $(x,t)\in Q_{T-s}$.
 Then
\begin{displaymath}
\!f_{sxx}\!(x\!,\!T\!-\!t\!,\!T\!-\!t)\!\le\!\mathbb{E}_{x\!,\!T\!-\!t}\!\bigg\{\!-\!\beta'\!(\!0\!)\tilde{F}''\!(\!X^{v}_{T}\!)\!+\!\int_{T\!-\!t}^{T}\!\Big[\!-\!\beta'\!(\!0\!)H_{xx}\!(\!X^{v}_{r}\!,\!r\!)\!+\!\mu_{xx}\!(\!X^{v}_{r}\!,\!r\!)w(\!X^{v}_{r}\!,\!T\!-\!r\!)\!\Big]\!dr\!\bigg\}\!.\!
\end{displaymath}
 and 
 \begin{displaymath}
 \lim\limits_{x\rightarrow-\infty}f_{sxx}(x,T-t,T-t)=0,\forall t\in[0,T-s].
 \end{displaymath}
 By (\ref{c8}) of Assumption \ref{all} and the maximum principle, we have 
 \begin{displaymath}
 f_{sxx}(x,T-t,T-t)\le H_{xx}(x,T-t),\forall(x,t)\in Q_{T-s}.
 \end{displaymath}
 Thus, $f\in\mathcal{R}_{0}$.
\end{proof}

\subsection{Existence of solutions to the original problem}\label{fixpoint}
In this subsection, we use the fixed point method to construct a solution to the HJB equations (\ref{HJB}).

Theorems \ref{exist-k}-\ref{exist-f} uniquely define a mapping $\mathcal{Y}_{2}$ that maps any $v\in\mathcal{R}_{1}$ to a solution of {\bf Problem 2} in $\mathcal{R}_{0}$. Let $\mathcal{Y}=\mathcal{Y}_{1}\circ\mathcal{Y}_{2}$, then $\mathcal{Y}$ is a mapping from $\mathcal{R}_{1}$ into itself.

\begin{theorem}
\label{exist-ori}
	The mapping $\mathcal{Y}$ has a fixed point $v$. As a corollary, $V(x,t):=\int_{-\infty}^{x}v(z,T-t)dz$ and $f^{s}(x,t)=(\mathcal{Y}_{2}v)(x,t)$ solve the HJB system (\ref{HJB}).
\end{theorem}
\begin{proof}
	We apply Schauder's fixed point theorem to find a fixed point of $\mathcal{Y}$ on $\mathcal{R}_{1}$. Equip $\mathcal{R}_{1}$ with the $\Vert\cdot\Vert_{1+\alpha,Q_{T}}$ norm given by
 \begin{equation*}
\Vert v\Vert_{1+\alpha,Q_{T}}:=\sum\limits_{m=1}^{\infty}\frac{1}{2^{m}}\min\{1,\Vert v\Vert_{C^{1+\alpha,\frac{1+\alpha}{2}}(\overline{Q^{(m)}})}\}.
\end{equation*}
where $Q^{(m)}=[a^{(m)},b^{(m)}]\times(t^{(m)},T],\forall m\ge 1$ is the same as in Theorem \ref{exist-v}. By definition, $\mathcal{R}_{1}$ is convex and closed. Hence the remaining proof boils down to showing that $\mathcal{Y}(\mathcal{R}_{1})$ is compact and that $\mathcal{Y}$ is continuous.

{\bf Step 1:} we show that $\mathcal{Y}(\mathcal{R}_{1})$ is compact.\\
Suppose  that $\{v^{(k)}\}_{k\ge 1}$ is a sequence in $\mathcal{R}_{1}$, we only need to show that there is a sub-sequence of $\{\mathcal{Y}v^{(k)}\}_{k\ge 1}$ converging in $\Vert\cdot\Vert_{C^{1+\alpha,\frac{1+\alpha}{2}}(\overline{Q^{(m)}})}$ for any $m\ge 1$.

Using interior $L^{p}$ estimate, we have that for any $v\in\mathcal{Y}_{1}(\mathcal{R}_{0})$,
	\begin{equation*}
	\Vert v\Vert_{W^{2,1}_{p}(Q^{(m)})}
	\le\tilde{C}\Big[\Vert H_{x}(x,T-t)\Vert_{L^{p}(Q^{(m+1)})}+1+\Vert c(T-t)\Vert_{L^{p}(Q^{(m+1)})}\Big],
	\end{equation*}
where the right-hand side is independent of the choice of $v$. Hence for $m=1$, there is a sub-sequence $\{\mathcal{Y}v^{(k_{j}^{(1)})}\}_{j\ge 1}$ of $\{\mathcal{Y}v^{(k)}\}_{k\ge 1}$ that converges to some $v^{\langle 1\rangle}$ weakly in $W^{2,1}_{p}(Q^{(1)})$, strongly in $C^{1+\alpha,\frac{1+\alpha}{2}}(\overline{Q^{(1)}})$. Then for $m=2$, there is a sub-sequence $\{\mathcal{Y}v^{(k_{j}^{(2)})}\}_{j\ge 1}$ of $\{\mathcal{Y}v^{(k_{j}^{(1)})}\}_{j\ge 1}$ that converges to some $v^{\langle 2\rangle}$ weakly in $W^{2,1}_{p}(Q^{(2)})$, strongly in $C^{1+\alpha,\frac{1+\alpha}{2}}(\overline{Q^{(2)}})$. Repeatedly, we can find $\{\mathcal{Y}v^{(k_{j}^{(m)})}\}_{j\ge 1}$ and $v^{\langle m\rangle}$ for any $m\ge 1$ satisfying the corresponding properties. Clearly, $v^{\langle m\rangle)},m\ge 1$ are consistent. Hence, we can define $v(x,t)=v^{\langle m\rangle}(x,t), \forall(x,t)\in Q^{(m)},m\ge 1$ and we have $v\in\mathcal{R}_{1}$. Then the diagonal sequence $\{\mathcal{Y}v^{(k_{j}^{(j)})}\}_{j\ge 1}$ is a sub-sequence of $\{\mathcal{Y}v^{(k)}\}_{k\ge 1}$ converging in $\Vert\cdot\Vert_{C^{1+\alpha,\frac{1+\alpha}{2}}(\overline{Q^{(m)}})}$ for any $m\ge 1$

{\bf Step 2:} We show that $\mathcal{Y}$ is continuous.\\
 Suppose that $u^{(k)}\rightarrow u^{*}$ is a converging sequence in $\mathcal{R}_{1}$, by Embedding theorem, we only need to prove that  $\mathcal{Y}u^{(k)}$ converges to $\mathcal{Y}u^{*}$ in $W^{2,1}_p(Q^{(m)})$ for each $m\ge 1$.

Define
\begin{align*}
&f^{(k)}=\mathcal{Y}_{2}u^{(k)},\ f^{*}=\mathcal{Y}_{2}u^{*},\ v^{(k)}=\mathcal{Y}_{1}f^{(k)},\ v^{*}=\mathcal{Y}_{1}f^{*},\\
& \Gamma^{(k)}(t)=\inf\{x|c(T-t)-u^{(k)}(x,t)=0\},\ \Gamma^{*}(t)=\inf\{x|c(T-t)-u^{*}(x,t)=0\}
\end{align*}
Let $[v^{(k)}]^{\epsilon,n},[v^{*}]^{\epsilon,n}$ be solutions to (\ref{vpb}) that converges to $v^{(k)},v^{*}$ in $W^{2,1}_{p}(Q^{(m)})$ for any $m\ge 1$, see Theorem \ref{exist-v}. Using interior $L^{p}$ estimate, we obtain
\begin{align*}
&\Vert v^{(k)}-v^{*}\Vert_{W^{2,1}_{p}(Q^{(m)})}\\
&\le C\bigg\{\Big\Vert f^{(k)}_{sx}(x,T-t,T-t)-f^{*}_{sx}(x,T-t,T-t)\Big\Vert_{L^{p}(Q^{(m+1)})}\\
&\!+\!\Big\Vert \alpha_{\epsilon,n}\big(c(T\!-\!t)\!-\![v^{(k)}]^{\epsilon,n}(x,T\!-\!t)\big)\!-\!\alpha_{\epsilon,n}\big(c(T\!-\!t)\!-\![v^{*}]^{\epsilon,n}(x,T\!-\!t)\big)\Big\Vert_{L^{p}(Q^{(m\!+\!1)})}\\
&\!+\!\Big\Vert [v^{(k)}]^{\epsilon,n}\!-\![v^{*}]^{\epsilon,n}\Big\Vert_{L^{p}(Q^{(\!m\!+\!1\!)})}\bigg\}\!+\!\Big\Vert v^{(k)}\!-\![v^{(k)}]^{\epsilon,n}\Big\Vert_{W^{2\!,\!1}_{p}(Q^{(\!m\!)})}\!+\!\Big\Vert v^{*}\!-\![v^{*}]^{\epsilon,n}\Big\Vert_{W^{2\!,\!1}_{p}(Q^{(\!m\!)})}.
\end{align*}
Recalling the definition of $\alpha_{\epsilon,n}$ and letting $\epsilon\rightarrow 0,n\rightarrow\infty$, we obtain
\begin{align}
\label{53eq1}
&\Vert v^{(k)}-v^{*}\Vert_{W^{2,1}_{p}(Q^{(m)})}\\
&\le \!C\!\bigg\{\!\Big\Vert f^{(\!k\!)}_{sx}(\!x\!,\!T\!-\!t\!,\!T\!-\!t\!)\!-\!f^{*}_{sx}(\!x\!,\!T\!-\!t\!,\!T\!-\!t\!)\Big\Vert_{L^{p}(Q^{(m\!+\!1)})}\!+\!\Big\Vert v^{(\!k\!)}\!-\!v^{*}\Big\Vert_{L^{p}(Q^{(m\!+\!1)})}\!\bigg\}\!.\notag
\end{align}
Note that $u^{(k)}\rightarrow u^{*}$ in the norm $\Vert \cdot\Vert_{1+\alpha,Q_{T}}$, we obtain that 
\begin{equation*}
\lim\limits_{k\rightarrow\infty}\Gamma^{(k)}(t)=\Gamma^{*}(t),\forall t\in[0,T].
\end{equation*}
Define
\begin{equation*}
w^{(k)}(x,t)=f^{(k)}_{sx}(x,T-t,T-t),\ w^{*}(x,t)=f^{*}_{sx}(x,T-t,T-t).
\end{equation*}
Then, for any $\epsilon>0$, for $k$ sufficiently large, the region $\big\{(x,t)\in Q_{T-s}\big||w^{(k)}(x,t)-w^{*}(x,t)|\ge\epsilon\big\}$ is bounded and belongs to $\big\{(x,t)\in Q_{T}\big|x<\Gamma^{(k)}(t),x<\Gamma^{*}(t)\big\}$. Applying the maximum principle to $\mathcal{L}$ in the region $\big\{(x,t)\in Q_{T}\big||w^{(k)}(x,t)-w^{*}(x,t)|>\epsilon\big\}$, we deduce that $|w^{(k)}(x,t)-w^{*}(x,t)|=\epsilon$ in the above region, and thus
\begin{equation*}
\lim\limits_{k\rightarrow\infty}\Vert w^{(k)}(x,t)-w^{*}(x,t)\Vert_{L^{\infty}(Q_{T})}=0.
\end{equation*}
Then we deduce that the first term in (\ref{53eq1}) vanishes as $k\rightarrow\infty$.

For the second term in (\ref{53eq1}), let $\epsilon>0$ as before, then for $k$ sufficiently large, we have $\Vert w^{(k)}(x,t)-w^{*}(x,t)\Vert_{L^{\infty}(Q_{T})}<\epsilon\min\limits_{(x,t)\in Q^{(m+1)}}\big[-\mu_{x}(x,t)\big]$. Then $\mathcal{L}v^{*}\le \mathcal{L}(v^{(k)}-\epsilon)$ if $v^{(k)}\ge v^{*}+\epsilon$ and $\mathcal{L}v^{(k)}\le \mathcal{L}(v^{*}-\epsilon)$ if $v^{*}\ge v^{(k)}+\epsilon$. Applying the maximum principle to $\mathcal{L}$ in $\big\{(x,t)\in Q^{(m+1)}\big||v^{(k)}(x,t)-v^{*}(x,t)|\ge\epsilon\big\}$, we deduce that $|v^{(k)}(x,t)-v^{*}(x,t)|=\epsilon$ in the above region, and thus
\begin{equation*}
\lim\limits_{k\rightarrow\infty}\Vert v^{(k)}(x,t)-v^{*}(x,t)\Vert_{L^{\infty}(Q^{(m+1)})}=0.
\end{equation*}
Then the second term in (\ref{53eq1}) vanishes as $k\rightarrow\infty$ and thus $\mathcal{Y}$ is continuous because $\Vert v^{(k)}-v^{*}\Vert_{W^{2,1}_{p}(Q^{(m)})}\rightarrow 0$ as $k\rightarrow\infty$ for any $m\ge 1$.
\end{proof}
 
 Finally, we have the existence of equilibria.
\begin{corollary}
	Assuming $\mu, \sigma, F, H, c$ are independent of $y$, $\sigma$ is bounded and Assumption \ref{all}, then there exists an equilibrium.
\end{corollary}
\begin{proof}
The proof follows from Theorem \ref{exist-ori} and Theorem \ref{verif}.
\end{proof}

\section{Conclusion}
\label{concl}
This work studies a time-inconsistent singular control problem where time-inconsistency is caused by non-exponential discount. Both the state dynamic and the objective functional are in general forms. We propose a novel definition of equilibrium. The definition in our paper provides a general and workable framework for further studies. Then, we establish both sufficient and necessary conditions to characterize the equilibrium. In particular, the equilibrium corresponds to a coupled PDE system with non-linearity, non-locality and free boundaries. Finally, we give a theoretical existence result of equilibrium solutions, which is new in the context of time-inconsistent singular control. Technically, we decouple the system into two sub-problems and apply Schauder's fixed point theorem. There are still some open problems. Further generalization of the existence result where $y$ is involved may require PDE theory beyond parabolic type equations. Also, the existence result only focuses on ``regular" equilibrium with equilibrium value function being $C^{2,1,1}$. From this perspective, any general non-existence result not restricted to regular equilibrium is beyond the PDE approach of this paper. We leave it as an open problem. Besides, uniqueness results of the equilibrium are also left for future research.

\section*{Acknowledgments}
The authors acknowledge the support from the National Natural Science Foundation of China (Grant No.12271290, and No.11871036). The authors also thank the members of the group of Actuarial Sciences and Mathematical Finance at the Department of Mathematical Sciences, Tsinghua University for their feedback and useful conversations. 

\bibliographystyle{plainnat}
\bibliography{references}
\end{document}